\documentclass[11pt]{article}

\usepackage{amsmath,amsthm}

\usepackage{amssymb,latexsym}

\usepackage{enumerate}

\newtheorem{tw}{Theorem}[subsection]
\newtheorem{lm}[tw]{Lemma}
\newtheorem{wn}[tw]{Corollary}
\newtheorem{stw}[tw]{Proposition}
\newenvironment{dow}{\it Proof.\rm}{\hfill $\Box$}

\theoremstyle{definition}
\newtheorem*{df}{Definition}
\newtheorem{uw}[tw]{Remark}
\newtheorem{prz}[tw]{Example}

\newcommand{\BN}{{\mathbb N}}
\newcommand{\BR}{{\mathbb R}}
\newcommand{\BX}{{\mathbb X}}

\newcommand{\FF}{{\mathcal{F}}}

\newcommand{\BB}{{\mathcal{B}}}

\newcommand{\DM}{{\mathcal{D}}}

\newcommand{\FD}{{\mathcal{FD}}}
\newcommand{\MM}{{\mathcal{M}}}

\newcommand{\TT}{{\mathcal{T}}}

\newcommand{\intt}{{\int_{t}^{T}}}
\newcommand{\into}{{\int_{0}^{T}}}
\newcommand{\intot}{{\int_{0}^{t}}}
\newcommand{\BRD}{{\mathbb{R}^{d}}}
\newcommand{\intz}{{\int_{t\wedge\zeta}^{T\wedge\zeta}}}
\newcommand{\intzt}{{\int_{t\wedge\tau\wedge\zeta}^{T\wedge\tau\wedge\zeta}}}
\newcommand{\VV}{{\mathcal{V}}}
\newcommand{\EE}{{\mathcal{E}}}
\newcommand{\sgn}{{\mbox{sgn}}}

\newcommand{\nsubsection}{\setcounter{equation}{0}\subsection}

\begin{document}

\title {Dirichlet forms and semilinear elliptic equations with measure data}
\author{Tomasz Klimsiak and Andrzej Rozkosz\footnote{Corresponding author.
Tel.: +48-56 611 2953; fax: +48-56 611 2987.}
\mbox{}\\[2mm]
{\small  Faculty of Mathematics and Computer Science,
Nicolaus Copernicus University},\\
{\small Chopina 12/18, 87-100 Toru\'n, Poland}}
\date{}
\maketitle
\begin{abstract}
We propose a probabilistic definition of solutions of semilinear
elliptic equations with (possibly nonlocal) operators associated
with regular Dirichlet forms and with measure data. Using the
theory of backward stochastic differential equations  we prove the
existence and uniqueness of solutions in the case where the
right-hand side of the equation is monotone and satisfies mild
integrability assumption, and the measure is smooth. We also study
regularity of solutions under the assumption that the measure is
smooth and has finite total variation. Some applications of our
general results are given.
\end{abstract}
{\em Keywords:} Semilinear elliptic equation, Measure data,
Dirichlet form, Backward sto\-chastic differential equation.
\medskip\\
{\em Mathematics Subject Classifications (2010):} Primary 35J61,
35R06; Secondary 60H30.

\footnotetext{{\em Email addresses:} tomas@mat.uni.torun.pl (T.
Klimsiak), rozkosz@mat.uni.torun.pl (A. Rozkosz). }
\nsubsection{Introduction}

Let $E$ be a locally compact separable metric space, $m$ be a
Borel measure on $E$ such that $\mbox{supp\,}[m]=E$, and let
$(\mathcal{E},D[\EE])$ be a regular Dirichlet form on
$L^{2}(E;m)$. Let $A$ denote the operator corresponding to
$(\mathcal{E},D[\EE])$, i.e. $A$ is a nonpositive self-adjoint
operator on $L^{2}(E;m)$ such that
\[
D(A)\subset D[\EE], \quad \EE(u,v)=(-Au,v),\,u\in D(A),v\in D[\EE]
\]
(see \cite{Fukushima}). In the present paper we investigate
semilinear elliptic equations of the form
\begin{equation}
\label{eq1.01} -Au=f(\cdot,u)+\mu,
\end{equation}
where $f:E\times\BR\rightarrow\BR$ is a measurable function and
$\mu$ is a smooth measure on $E$. Equations of the form
(\ref{eq1.01}) include semilinear equations for local operators
(the model example is the Laplace operator subject to the
Dirichlet or Neumann boundary conditions) as well as for nonlocal
operators (the model example is the fractional Laplacian).

There are many papers devoted to equations of the form
(\ref{eq1.01}) in case $A$ is an elliptic second-order operator in
divergence form and $\mu$ is a Radon measure (see, e.g.,
\cite{BBGGPV,BGO,DMOP,MP} and the references given there). One of
the main problems one encounters when considering such equations
is to give proper definition of solutions which ensures
uniqueness. To tackle this problem the so-called renormalized
solutions (see \cite{BBGGPV,DMOP,MP}) and entropy solutions (see
\cite{BGO}) were introduced. Roughly speaking, these solutions are
measurable functions whose truncates belong to the energy space,
which satisfy an estimate on the decay of their energy on sets
where their are large and satisfy (\ref{eq1.01}) in the
distributional sense for some wide class of test functions.

Our approach to (\ref{eq1.01}) is quite different. In the paper we
consider generalized probabilistic solutions of the problem
(\ref{eq1.01}). Let $S$ denote the class of all smooth measures on
$E$ (see Section \ref{sec4} for the definition; in particular
every soft measure (see \cite{DPP}) belongs to $S$).
 We
first prove that if $\mu\in S$ and $f$ is continuous and monotone
with respect to the second variable and satisfies some mild
integrability assumptions then the probabilistic solution of
(\ref{eq1.01}) exists and is unique in some class of functions
having weak regularity properties. Then we show that if $\mu$
belongs to the class $\MM_{0,b}$ of smooth measures of finite
total variation and the form $(\EE,D[\EE])$ is transient then the
solution has additional regularity properties.

To be more specific, let us denote by
$\BX=(\Omega,\FF,\FF_t,X,P_x)$ a Hunt process with life-time
$\zeta$ associated with the form $(\EE,D[\EE])$ and let $A^{\mu}$
denote the continuous additive functional of $\BX$ which is in the
Revuz correspondence with $\mu\in S$ (see \cite{Fukushima}). By a
probabilistic solution of (\ref{eq1.01}) we mean a
quasi-continuous function $u:E\rightarrow\BR$ such that
\begin{equation}
\label{eq1.02} u(x)=E_{x}\int_{0}^{\zeta}f(X_{t},u(X_{t}))\,dt
+E_{x}\int_{0}^{\zeta}dA^{\mu}_{t}
\end{equation}
for q.e. $x\in E$, i.e. $u$ satisfies the nonlinear Feynman-Kac
formula naturally associated with $(\EE,D[\EE])$ and $\mu,f$. In
the main theorem we prove that if $\mu\in S$, $f$ satisfy the
assumptions
\begin{enumerate}
\item[(A1)]$f:E\times\BR\rightarrow\BR$ is measurable and
$y\mapsto f(x,y)$ is continuous for every $x\in E$,
\item [(A2)]$(f(x,y_{1})-f(x,y_{2}))(y_{1}-y_{2})\le 0$ for every
$y_{1}, y_{2}\in \BR$ and $x\in E$,
\item [(A3$'$)]for every $r>0$ the function
$F_r(x)=\sup_{|y|\le r}|f(x,y)|$, $x\in E$, is quasi-$L^{1}$ with
respect to $(\EE,D[\EE])$, i.e. $t\mapsto F_r(X_t)$ belongs to
$L^1_{loc}(\BR_+)$ $P_x$-a.s. for q.e. $x\in E$,
\item [(A4$'$)]$E_{x}\int_{0}^{\zeta}|f(X_{t},0)|\,dt<\infty$,
$E_{x}\int_{0}^{\zeta}d|A^{\mu}|_{t}<\infty$ for $m$-a.e. $x\in
E$,
\end{enumerate}
then there exists a unique solution of (\ref{eq1.02}) in the class
of quasi-continuous functions $u:E\rightarrow\BR$ such that the
process $t\mapsto u(X_t)$ is of Doob's class (D) under the measure
$P_x$ for q.e. $x\in E$. Moreover, for every $q\in(0,1)$,
$E_x\sup_{t\ge0}|u(X_t)|^q<\infty$ for q.e. $x\in E$. We also show
that (A3$'$) is implied by (A3) and if $(\EE,D[\EE])$ is transient
then (A4$'$) is implied by (A4), where
\begin{enumerate}
\item[(A3)]$F_r\in L^1(E;m)$,
\item [(A4)]$f(\cdot,0)\in L^1(E;m)$, $\mu\in\MM_{0,b}$.
\end{enumerate}

Let us remark that (A3$'$), (A4$'$) are the minimal conditions
which make it possible to define solutions of (\ref{eq1.01}) by
(\ref{eq1.02}). Conditions (A1)--(A4) are widely used in
$L^1$-theory of nonlinear elliptic equations (see, e.g.,
\cite{BBGGPV}).

We have already mentioned that transiency of $(\EE,D[\EE])$ and
additional assumptions on  $\mu$ imply better regularity
properties of the solution of (\ref{eq1.02}). Namely, for
transient forms, if $u$ is a solution of (\ref{eq1.02}) with
$\mu,f$ satisfying (A4) then $f_u\in L^1(E;m)$ and
\[
\|f_{u}\|_{L^{1}(E;m)}
\le\|f(\cdot,0)\|_{L^{1}(E;m)}+\|\mu\|_{TV},
\]
where $\|\mu\|_{TV}$ is the total variation norm of $\mu$.
Moreover, for every $k>0$ the truncation of $u$ defined by
$T_k(u)=\min\{k,\max\{-k,u\}\}$ belongs to the extended Dirichlet
space $\FF_e$ of $(\EE,D[\EE])$ and
\[
\EE(T_{k}(u),T_{k}(u))\le k(\|f_{u}\|_{L^{1}(E;m)}+\|\mu\|_{TV})
\]
as well as
\[
\EE(\Phi_{k}(u),\Phi_{k}(u))\le\int_{\{|u|\ge k\}}
|f_{u}(x)|\,m(dx) +\int_{\{|u|\ge k\}}\,d|\mu|,
\]
where $\Phi_{k}(u)=T_{1}(u-T_{k}(u))$.
These estimates are analogues of energy estimates for renormalized
solutions. Up to now they were known for some classes of local
operators (see, e.g., \cite{BBGGPV}). In general, $u$ is even not
locally integrable. We show that nevertheless $u\in L^1(E;m)$ in
many interesting situations.

Another remarkable feature of  probabilistic solutions in the
transient case is that for $\mu\in\MM_{0,b}$ they can be defined
in purely analytic way, which resembles Stampacchia's way to
defining solutions. Let $S^{(0)}_0$ denote the set of nonnegative
Radon measures on $E$ of finite $0$-order energy integral and let
$S^{(0)}_{00}$ be the subset of $S^{(0)}_0$ consisting of finite
measures $\mu$ such that $\|U\mu\|_{\infty}<\infty$, where $U\mu$
is the (0-order) potential of $\mu$ (see \cite{Fukushima}). We
show that if $\mu\in\MM_{0,b}$, $u$ is quasi-continuous and
$f(\cdot,u)\in L^1(E;m)$ then $u$ is a probabilistic solution of
(\ref{eq1.01}) if and only if $u$ is a solution of (\ref{eq1.01})
in the sense of duality, i.e.
$|\langle\nu,u\rangle|=|\int_Eu\,d\nu|<\infty$ for every $\nu\in
S^{(0)}_{00}$ and
\[
\langle\nu,u\rangle=(f(\cdot,u),U\nu)_{L^{2}(E;m)}+\langle\mu,U\nu\rangle,
\quad \nu\in S^{(0)}_{00},
\]
If, in addition, $\mu\in S^{(0)}_0$ and $f(\cdot,u)\cdot m\in
S^{(0)}_0$, then $u$ is a weak solution of (\ref{eq1.01}) in the
sense that $u$ belongs to the extended Dirichlet space $\FF_e$ and
\[
\EE(u,v)=(f(\cdot,u),v)_{L^2(E;m)}+\langle v,\mu\rangle,\quad
v\in\FF_e.
\]

To apply our general results to concrete operator, one has to
check that the form corresponding to it is a regular Dirichlet
form and, to get better regularity of solutions, that the form is
transient. In the paper we recall two classical examples of local
and nonlocal operators associated with such forms, namely
divergence form operators and L\'evy diffusion generators. In the
latter case our results lead to theorems on existence, uniqueness
and regularity of equations of the form
\[
-\psi(\nabla)u=f(\cdot,u)+\mu, \quad u_{|D^{c}}=0,
\]
where $D$ is an open subset of $\BR^d$ and $\psi$ is the
L\'evy-Khintchine symbol of some symmetric convolution semigroup
of measures on $\BR^d$. These theorems are new in the theory of
semilinear equations with measure data. Note, however, that linear
equations with fractional Laplacian and bounded smooth measure on
the right-hand side are considered in \cite{KarlsenPetittaUlusoy}.
The first example is provided mainly to illustrate that our
approach allows one to treat in a unified way many interesting
operators. It should be stressed, however, that even in the case
of divergence form operators our results are new, because
probabilistic approach enables us treat equations with measures
which are not necessarily Radon measures. To our knowledge, our
results for equations with Radon measures and possibly
degenerating operator are also new. Some other possible
applications of the main results of the paper are briefly
indicated in Section \ref{sec6}.

Our proof of the main result on existence and uniqueness of
solutions of (\ref{eq1.02}) is probabilistic in nature. The idea
is as follows. First we show  that there exists a progressively
measurable process $Y$ of class (D) and a martingale $M$ such that
$Y_{T\wedge\zeta}\rightarrow0$ as $T\rightarrow+\infty$ and for
every $T>0$,
\begin{equation}
\label{eq1.03}
Y_t=Y_{T\wedge\zeta}+\int^{T\wedge\zeta}_{t\wedge\zeta}f(X_s,Y_s)\,ds
+\int^{T\wedge\zeta}_{t\wedge\zeta}dA^{\mu}_s
-\int^{T\wedge\zeta}_{t\wedge\zeta}dM_s,\quad t\in[0,T].
\end{equation}
Then we set
\begin{equation}
\label{eq1.04} u(x)=E_xY_0,\quad x\in E
\end{equation}
and show that $u$ is quasi-continuous. Finally, using the Markov
property we show that $u(X_t)=Y_t$, $t\ge0$, $P_x$-a.s. for q.e.
$x\in E$, which leads to (\ref{eq1.02}). Let us point out that
(\ref{eq1.04}) means that the solution $u$ of (\ref{eq1.02}) is
given by the first component of the solution $(Y,M)$ of the
backward stochastic differential equation (\ref{eq1.03}). This
representation is useful. For instance, it allows one to prove
easily the comparison theorem for solutions of (\ref{eq1.01}) and
show that the solutions have some integrability properties.

The rest of the paper is organized as follows. In Sections
\ref{sec2} and \ref{sec3} we prove  theorems on existence,
uniqueness and comparison of $L^p$-solutions of some general
(non-Markovian) backward stochastic differential equations
(BSDEs). In Section \ref{sec4} we prove our main result on
existence and uniqueness of probabilistic solutions of
(\ref{eq1.01}) in case $\mu\in S$. In Section \ref{sec5} we
investigate regularity of probabilistic solutions of (\ref{eq1.1})
under the additional assumptions that $\EE$ is transient and
$\mu\in\MM_{0,b}$.  In Section \ref{sec6} some applications of
general theorems proved in Sections \ref{sec4} and {\ref{sec5} are
given.

\nsubsection{Generalized BSDEs with constant terminal time}
\label{sec2}

We assume as given a complete probability space $(\Omega,\FF,P)$
equipped with a complete right continuous filtration $\{\FF_{t},
t\ge 0\}$.

$\mathcal{S}$ (resp. $\DM$) is the space of all progressively
measurable continuous (resp. c\`adl\`ag) processes.
$\mathcal{S}^{p}$ (resp. $\DM^{p}$), $p>0$, is  the space of all
processes $X\in\mathcal{S}$ (resp. $X\in\DM$) such that
\[
E\sup_{t\ge 0}|X_{t}|^{p}<\infty.
\]

$\MM$ (resp. $\MM_{loc}$) is the space of all c\`adl\`ag
martingales (resp. c\`adl\`ag local martingales) and $\MM^{p}$,
$p>0$, is the subspace of $\MM$ consisting of all martingales such
that $E([M]_{\infty})^{p/2}<\infty$.

$\mathcal{V}$ is the space of all c\`adl\`ag  progressively
measurable processes of finite variation such that $V_{0}=0$. If
$V\in\mathcal{V}$ then by $|V|_t$ we denote the variation of $V$
on $[0,t]$ and by $dV$ the random measure generated by the
trajectories of $V$.

By $\mathcal{T}$ we denote the set of all finite stopping times
and by $\mathcal{T}_t$ the set of all stopping times with values
in $[0,t]$. We recall that a c\`adl\`ag adapted process $Y$ is
said to be of class (D) if the collection
$\{Y_{\tau},\tau\in\TT\}$ is uniformly integrable. For a process
$Y$ of class (D) we set
\[
\|Y\|_{1}=\sup\{E|Y_{\tau}|,\tau\in\mathcal{T}\}.
\]

For a  process $X\in\DM$ we set $X_{t-}= \lim_{s\nearrow t} X_{s}$
and $\Delta X_{t}= X_{t}-X_{t-}$ with the convention that
$X_{0-}=0$. Let $\{X^{n}\}\subset \DM$, $X\in\DM$. We say that
$X^{n}\rightarrow X$ in ucp (uniformly on compacts in probability)
if $\sup_{t\in [0,T]}|X^{n}_{t}-X_{t}|\rightarrow 0$ in
probability $P$ for every $T>0$.

In the whole paper all equalities and inequalities and other
relations between random elements are understood to hold $P$-a.s.
To avoid ambiguity we stress that writing $X_{t}=Y_{t},\, t\in
[0,T]$ we mean that $X_{t}=Y_{t}$, $t\in [0,T]$, $P$-a.s. whereas
writing $X_{t}=Y_{t}$ for a.e. (resp. for every) $t\in [0,T]$ we
mean that $X_{t}=Y_{t}$, $P$-a.s. for a.e. (resp. for every) $t\in
[0,T]$. We also adopt the convention that
$\int_{a}^{b}=\int_{(a,b]}$.

$T_{k}(x)=\min\{k,\max\{-k,x\}\}$, $x\in\BR$. $x^{+}=\max\{x,0\}$,
$x^{-}=\max\{-x,0\}$ and
\[
\hat x=\hat{\mbox{sgn}}(x),\quad
\hat{\mbox{sgn}}(x)=\mathbf{1}_{x\neq 0}\frac{x}{|x|},\quad
x\in\mathbb{R}^d.
\]

\begin{df}
Let $\xi\in\FF_{T}$, $V\in \mathcal{V}$ and let
$f:[0,T]\times\Omega\times\mathbb{R}\rightarrow \mathbb{R}$ be a
function such that $f(\cdot,y)$ is progressively measurable for
every $y\in\mathbb{R}$. We say that a pair $(Y,M)$ is a solution
of BSDE$(\xi,f+dV)$ on $[0,T]$ if $Y\in\DM,\, M\in
\mathcal{M}_{loc}$, $t\mapsto f(t,Y_{t})\in L^{1}(0,T)$ and
\begin{equation}
\label{eq2.14} Y_{t}=\xi+\intt f(s,Y_{s})\,ds+\intt dV_{s}-\intt
dM_{s},\quad t\in [0,T].
\end{equation}
\end{df}

We will need the following hypotheses.
\begin{enumerate}
\item [(H1)] For every $t\in [0,T]$ the mapping
$\mathbb{R}\ni y\mapsto f(t,y)$ is continuous.
\item[(H2)]$(f(t,y)-f(t,y'))(y-y')\le 0$ for every $t\ge0$, $y,y'\in\BR$.
\item[(H3)] For every $r>0$ the mapping
$[0,T]\ni t\mapsto\sup_{|y|\le r}|f(t,y)-f(t,0)|$ belongs to
$L^{1}(0,T)$.
\item[(H4)]$E|\xi|^{p}+E(\int_{0}^{T}|f(t,0)|\,dt)^{p}+E(\into
d|V|_{t})^{p}<\infty$.
\item[(A)]There exists a nonnegative progressively measurable
process $\{f_{t}\}$ such that
\[
\forall (t,y)\in [0,T]\times\mathbb{R},
\quad \hat{y}f(t,y)\le f_{t}.
\]
\end{enumerate}

Uniqueness of solutions of (\ref{eq2.14}) follows from the
following comparison result.
\begin{stw}
\label{lm0.1}  Let $(Y^{1},M^{1}), (Y^{2},M^{2})$ be solutions  of
BSDE$(\xi^{1},f^{1}+dV^{1})$ and  BSDE$(\xi^{2},f^{2}+dV^{2})$,
respectively, such that $Y^{1}, Y^{2}$ are of class \mbox{\rm(D)}.
Assume that $\xi^{1}\le \xi^{2}$, $dV^{1}\le dV^{2}$ and that
\begin{equation}
\label{eq2.01} f^{2}\mbox{ satisfies {\rm(H2)} and
}f^{1}(t,Y^{1}_{t})\le f^{2}(t,Y^{1}_{t})\mbox{ for a.e. }t\in
[0,T]
\end{equation}
or
\begin{equation}
\label{eq2.02} f^{1}\mbox{ satisfies {\rm(H2)} and
}f^{1}(t,Y^{2}_{t})\le f^{2}(t,Y^{2}_{t})\mbox{ for a.e. }t\in
[0,T].
\end{equation}
Then
\[
Y^{1}_{t}\le Y^{2}_{t},\quad t\in [0,T].
\]
\end{stw}
\begin{dow}
We give the proof in  case (\ref{eq2.01}) is satisfied. In case
(\ref{eq2.02}) is satisfied the proof is analogous and hence left
to the reader. Let $\tau\in\mathcal{T}_{T}$. By the It\^o-Tanaka
formula,
\begin{align*}
(Y^{1}_{t\wedge\tau}-Y^{2}_{t\wedge\tau})^{+}
&\le (Y^{1}_{\tau}-Y^{2}_{\tau})^{+}
+\int_{t\wedge\tau}^{\tau}
\mathbf{1}_{\{Y^{1}_{s-}>Y^{2}_{s-}\}}
(f^{1}(s,Y^{1}_{s})-f^{2}(s,Y^{2}_{s}))\,ds\\
&\quad+\int_{t\wedge\tau}^{\tau}
\mathbf{1}_{\{Y^{1}_{s-}>Y^{2}_{s-}\}}\,d(V^{1}_{s}-V^{2}_{s})
-\int_{t\wedge\tau}^{\tau}\mathbf{1}_{\{Y^{1}_{s-}>Y^{2}_{s-}\}}
\,d(M^{1}_{s}-M^{2}_{s}).
\end{align*}
From the above and the assumptions,
\[
(Y^{1}_{t\wedge\tau}-Y^{2}_{t\wedge\tau})^{+} \le
(Y^{1}_{\tau}-Y^{2}_{\tau})^{+}
-\int_{t\wedge\tau}^{\tau}\mathbf{1}_{\{Y^{1}_{s-}>Y^{2}_{s-}\}}\,
d(M^{1}_{s}-M^{2}_{s}),\quad t\in [0,T].
\]
Let $\{\tau_{k}\}$ be a fundamental sequence for the local
martingale $M^{1}-M^{2}$. Since $Y^{1},Y^{2}$ are of class (D),
taking expectation of both sides of the above inequality with
$\tau$ replaced by $\tau_{k}$ and then letting
$k\rightarrow\infty$ show that $E(Y^{1}_{t}-Y^{2}_{t})^{+}\le 0$,
$t\in[0,T]$. This proves the proposition since $Y^{1}, Y^{2}$ are
c\`adl\`ag processes.
\end{dow}

\begin{wn}
\label{cor2.2} Assume \mbox{\rm(H2)}. Then there exists at most
one solution $(Y,M)$ of BSDE$(\xi,f+dV)$ such that $Y$ is of class
\mbox{\rm(D)}.
\end{wn}

The following a priori estimates will be needed in the proof of
existence of solutions of (\ref{eq2.14}).
\begin{lm}
\label{lm0.2} Let $p>0$ and let $(Y,M)$ be a solution of
BSDE$(\xi,f+dV)$ such that $(Y,M)\in\DM^{p}\otimes\MM^{p}$ if
$p\neq1$ and $Y$ is of class \mbox{\rm(D)}, $M\in\MM_{loc}$ if
$p=1$. If \mbox{\rm (H2), (H4)} are satisfied then
\begin{align*}
E(\int_{0}^{T}|f(t,Y_{t})|\,dt)^{p} \le
c_{p}E\left(|\xi|^{p}+(\into|f(t,0)|\,dt)^{p} +(\into
d|V|_{t})^{p}+\mathbf{1}_{\{p\neq1\}}[M]^{p/2}_{T}\right).
\end{align*}
\end{lm}
\begin{dow}
Let $\tau\in\TT_{T}$. By the It\^o-Tanaka formula,
\begin{equation}
\label{eq2.03}
-\int_{0}^{\tau}\hat{\mbox{sgn}}(Y_{t})f(t,Y_{t})\,dt \le
|Y_{\tau}|-|Y_{0}|-\int_{0}^{\tau}\hat\sgn(Y_{t-})\,dM_{t}
+\int_{0}^{\tau}d|V|_{t}.
\end{equation}
By (H3),
\[
0\le -\int_{0}^{\tau}\hat\sgn(Y_{t})(f(t,Y_{t})-f(t,0))\,dt.
\]
Combining this with (\ref{eq2.03}) we get
\[
\int_{0}^{\tau}|f(t,Y_{t})|\,dt\le \int_{0}^{\tau}|f(t,0)|\,dt
+|Y_{\tau}|-\int_{0}^{\tau}\hat\sgn(Y_{t-})\,dM_{t}
+\int_{0}^{\tau}d|V|_{t}\,,
\]
from which one can easily deduce the desired inequality.
\end{dow}

\begin{uw}
In case $p\neq1$ the statement of Lemma \ref{lm0.2} remains valid
if we replace the  condition $Y\in\DM^p$ by the condition that
$|Y|^{p}$ is of class (D).
\end{uw}

\begin{lm}
\label{lm0.3} Let $p>0$ and let $(Y,M)$ be a solution of
BSDE$(\xi,f+dV)$. Assume that \mbox{\rm(A)} is satisfied and that
\[
E(\into f_{t}\,dt)^{p}+E(\into d|V|_{t})^{p}<\infty,\quad
E\sup_{0\le t\le T}|Y_{t}|^{p}<\infty.
\]
If $p\in (0,2]$ or $p>2$ and $M$ is locally in $\MM^{p}$, then
\[
E[M]^{p/2}_{T}\le c_{p}E\left(\sup_{0\le t\le T}
|Y_{t}|^{p}+(\into f_{t}\,dt)^{p}+(\into d|V|_{t})^{p}\right).
\]
\end{lm}
\begin{dow}
Let $\tau\in\TT_{T}$. By It\^o's formula,
\begin{align}
\label{eq.fv1} |Y_{0}|^{2}+[M]_{\tau}=|Y_{\tau}|^{2}
+2\int_{0}^{\tau} Y_{t}f(t,Y_{t})\,dt
+\int_{0}^{\tau}Y_{t-}\,dV_{t} -2\int_{0}^{\tau}Y_{t-}\,dM_{t}.
\end{align}
By the above and (A),
\begin{align*}
[M]_{\tau}\le \sup_{0\le t\le T}|Y_{t}|^{2} +2\sup_{0\le t\le T}
|Y_{t}|\into f_{t}\,dt +2\sup_{0\le t\le T}|Y_{t}|\into
d|V|_{t}-2\int_{0}^{\tau}Y_{t-}\,dM_{t}.
\end{align*}
By Young's inequality,
\begin{align}
\label{eq0.1} [M]^{p/2}_{\tau}\le b_{p} \left(\sup_{0\le t\le T}
|Y_{t}|^{p}+(\into f_{t}\,dt)^{p} +(\into d|V|_{t})^{p}
+|\int_{0}^{\tau} Y_{t-}\,dM_{t}|^{p/2}\right).
\end{align}
Suppose that $E[M]^{p/2}_{\tau}<\infty$ for some $\tau\in\TT_{T}$.
Then by the Burkholder-Davis-Gundy inequality, It\^o's isometry
and again Young's inequality,
\begin{align*}
b_{p}E|\int_{0}^{\tau}Y_{t-}\,dM_{t}|^{p/2}& \le
c_{p}E[\int_{0}^{\cdot}Y_{t-}\,dM_{t}]_{\tau}^{p/4}
=c_{p}E(\int_{0}^{\tau}Y_{t-}^{2}\,d[M]_{t})^{p/4}\\
&\le c_{p}E(\sup_{0\le t\le T}|Y_{t}|^{p/2}[M]_{\tau}^{p/4}) \le
\frac{c^{2}_{p}}{2} E\sup_{0\le t\le T}|Y_{t}|^{p}+\frac12
E[M]_{\tau}^{p/2}.
\end{align*}
Combining this with (\ref{eq0.1}) gives
\begin{align}
\label{sotsop}
E[M]^{p/2}_{\tau}\le d_{p}E\left(\sup_{0\le t\le T}
|Y_{t}|^{p}+(\into f_{t}\,dt)^{p}+(\into d|V|_{t})^{p}\right).
\end{align}
To complete the proof it is enough to show that for every $p>0$
there exists a stationary sequence $\{\tau_{k}\}\subset\TT_{T}$
such that $M^{\tau_{k}}\in\MM^{p}$, because then (\ref{sotsop})
holds true with $\tau$ replaced by $\tau_k$, so letting
$k\rightarrow\infty$ and using Fatou's lemma we obtain the
required inequality. If $p>2$ then the existence of $\{\tau_k\}$
follows from the assumption on $M$. If $p\in(0,2]$ then any
fundamental sequence for the local martingale $\int_{0}^{\cdot}
Y_{t-}\,dM_{t}$ has the desired property. Indeed, if $\{\tau_k\}$
is such a sequence then by (\ref{eq.fv1}),
\[
E[M]^{p/2}_{\tau_{k}} \le c E\left(\sup_{0\le t\le T}
|Y_{t}|^{p}+(\int_{0}^{T}|f(t,0)|\,dt)^{p}
+|\int_{0}^{\tau_{k}}Y_{t-}\,dM_{t}|\right)
\]
and the right-hand side of the above inequality is finite by the
assumptions of the lemma and the very definition of the
fundamental sequence.
\end{dow}

\begin{lm}
\label{lm0.4} Assume that \mbox{\rm(H1)--(H3)} are satisfied and
there exists $C>0$ such that $\sup_{0\le t\le
T}|f(t,0)|+|V|_{T}+|\xi|\le C$. Then there exists a unique
solution $(Y,M)\in\DM^{2}\otimes\MM^{2}$ of BSDE$(\xi,f+dV)$.
\end{lm}
\begin{dow}
We first assume additionally that there is $L>0$ such that
\begin{equation}
\label{eq2.05} |f(t,y)-f(t,y')|\le L|y-y'|
\end{equation}
for $t\in[0,T]$, $y,y'\in\BR$. For $U\in\DM^{2}$ let $Y^{U},M^{U}$
denote c\`adl\`ag versions of the processes $\tilde Y^{U},\tilde
M^{U}$ defined by
\[
\tilde Y^{U}_{t}=E(\xi+\into f(s,U_{s})+\into
dV_{s}|\FF_{t})-\intot f(s,U_{s})\,ds -\intot dV_{s}
\]
and
\[
\tilde M^{U}_{t}=E(\xi+\into f(s,U_{s})+\into
dV_{s}|\FF_{t})-\tilde Y^{U}_{0}.
\]
Then $(Y^{U},M^{U})$ is a unique solution, in the class
$\DM^{2}\otimes\MM^{2}$, of the BSDE
\begin{equation}
\label{eq0.2}
Y^{U}_{t}=\xi+\intt f(s,U_{s})\,ds+\intt dV_{s}-\intt dM^{U}_{s},
\quad t\in [0,T].
\end{equation}
Therefore we may define the mapping $\Phi:\DM^{2}\otimes
\MM^{2}\rightarrow\DM^{2}\otimes\MM^{2}$ by putting
\[
\Phi (U,N)=(Y^{U},M^{U}).
\]
By standards arguments (see, e.g., the proof of \cite[Proposition
2.4 ]{Pardoux}) one can show that $\Phi$ is contractive on the
Banach space $(\DM^{2}\otimes\MM^{2}, \|\cdot\|_{\lambda})$, where
\[
\|(Y,M)\|_{\lambda}=E\sup_{0\le t\le T}e^{\lambda t}|Y_{t}|^{2}
+E[\int_{0}^{\cdot}e^{\lambda t}\,dM_{t}]_{T}
\]
with suitably chosen $\lambda>0$. Consequently, $\Phi$ has a fixed
point $(Y,M)\in\DM^{2}\otimes\MM^{2}$. Obviously $(Y,M)$ is a
unique solution of BSDE$(\xi,f+V)$. We now show how to dispense
with the assumption (\ref{eq2.05}). For $n\in\BN$ put
\[
f_{n}(t,y)=\inf_{x\in\mathbb{Q}}\{n|y-x|+f(t,x)\}.
\]
It is an elementary check that
\begin{enumerate}
\item [(a)]$|f_{n}(t,0)|\le C$, $|f_{n}(t,y)-f_{n}(t,y')|\le n|y-y'|$ for
all $t\in [0,T]$, $y,y'\in\mathbb{R}$,
\item [(b)]$f_{1}(t,y)\le f_{n}(t,y)\le f(t,y)$ for all $t\in
[0,T]$, $y\in\mathbb{R}$ and $f_{n}(t,\cdot)\nearrow f(t,\cdot)$
uniformly on compact subsets of $\mathbb{R}$,
\item [(c)]$\sup_{|y|\le r}|f_{n}(t,y)|\le r+C+\sup_{|y|\le r}|f(t,y)|$
for every $r>0$ and $f_n$ satisfies (H2).
\end{enumerate}
By what has already been proved, for each $n\in\BN$ there exists a
unique solution $(Y^{n}, M^{n})\in\DM^{2}\otimes\MM^{2}$ of
BSDE$(\xi,f_{n}+dV)$. By the It\^o--Tanaka formula and (H2),
\begin{align*}
|Y^{n}_{t}|&\le |\xi|
+\intt\hat\sgn(Y^{n}_{s})f_n(s,Y^{n}_{s})\,ds
+\intt\hat\sgn(Y_{s-})\,dV_{s}-\intt\hat\sgn(Y^{n}_{s-})\,dM_{s}\\
&\le |\xi|+\into |f_n(s,0)|\,ds +\into
d|V|_{s}-\intt\hat\sgn(Y^{n}_{s-})\,dM_{s}.
\end{align*}
By the above and the assumptions on $\xi,f,V$,
\begin{equation}
\label{eq0.3}
|Y^{n}_{t}|\le E(|\xi|+\into |f(s,0)|\,ds+\into d|V|_{s}|\FF_{t})\le C.
\end{equation}
Moreover, by Proposition \ref{lm0.1}, $Y^{n}_{t}\le Y^{n+1}_{t}$,
$t\in [0,T]$. Therefore defining $Y_{t}=\sup_{n\ge 1} Y^{n}_{t}$,
$t\in [0,T]$, we see that
\begin{equation}
\label{eq0.5} Y^{n}_{t}\nearrow Y_{t},\quad t\in [0,T], \qquad
E\into|Y^{n}_{t}-Y_{t}|^{p}\,dt\rightarrow 0,\quad p\ge0.
\end{equation}
By (\ref{eq0.3}), (\ref{eq0.5}), (H3) and (b), (c),
$\into|f_{n}(t,Y^{n}_{t})-f(t,Y_{t})|\,dt\rightarrow0$, while by
Lemmas \ref{lm0.2} and \ref{lm0.3}, $\sup_{n\ge 1}
E(\into|f_{n}(t,Y^{n}_{t})|\,dt)^{2}<\infty$. Hence
\begin{equation}
\label{eq0.7} E(\into
|f_{n}(t,Y^{n}_{t})-f(t,Y_{t})|\,dt)^{p}\rightarrow 0
\end{equation}
for every $p\in (1,2)$. Next, by Doob's $L^{p}$-inequality,
\begin{align}
\label{eq0.8}
\nonumber E\sup_{0\le t\le T}|Y^{n}_{t}-Y^{m}_{t}|^{p}
&\le E\sup_{0\le t\le T}\left(E(\into|f_{n}(s,Y^{n}_{s})-f_{m}(s,Y^{m}_{s})|\,ds
|\FF_{t})\right)^{p}\\
&\le c(p) E(\into |f_{n}(s,Y^{n}_{s})-f_{m}(s,Y^{m}_{s})|\,ds)^{p},
\end{align}
which when combined with (\ref{eq0.3})--(\ref{eq0.7}) shows that
$Y\in\DM^{2}$ and $Y^{n}\rightarrow Y$ in $\DM^{p},\, p\in (1,2)$.
Since
\[
Y^{n}_{t}=E(\xi+\intt f_{n}(s,Y^{n}_{s})\,ds +\intt
dV_{s}|\FF_{t}),
\]
using the fact that $Y^{n}\rightarrow Y$ in $\DM^{p},\,p\in(1,2)$ and
(\ref{eq0.7}) we conclude that
\begin{equation}
\label{eq2.11} Y_{t}=E(\xi+\intt f(s,Y_{s})\,ds+\intt
dV_{s}|\FF_{t}),\quad t\in[0,T].
\end{equation}
Therefore the pair $(Y,M)$, where
\begin{equation}
\label{eq2.12} M_{t}=E(\xi+\into f(s,Y_{s})\,ds +\into
dV_{s}|\FF_{t})-Y_{0},\quad t\in[0,T]
\end{equation}
is a solution of BSDE$(\xi,f+dV)$. The desired integrability
properties of $(Y,M)$ follow immediately from
Lemma \ref{lm0.3}.
\end{dow}

\begin{tw}
\label{th0.1} Assume that \mbox{\rm(H1)--(H3)} and \mbox{\rm(H4)}
with $p=1$ are satisfied. Then there exists a unique solution
$(Y,M)$ of BSDE$(\xi,f+dV)$ such that $(Y,M)\in
\DM^{q}\otimes\MM^{q}$, $q\in(0,1)$, $M$ is uniformly integrable
and $Y$ is of class \mbox{\rm(D)}.
\end{tw}
\begin{dow}
Write
\[
\xi^{n}=T_{n}(\xi),\quad f_{n}(t,y)=f(t,y)-f(t,0)+T_{n}(f(t,0)),
\quad  V^{n}_{t}=\intot \mathbf{1}_{\{|V|_{s}\le n\}}\,dV_{s}\,.
\]
By Lemma \ref{lm0.4}, for each $n\in\BN$ there exists a unique
solution $(Y^{n},M^{n})\in\DM^{2}\otimes\MM^{2}$ of
BSDE$(\xi^{n},f_{n}+dV^{n})$. In particular,
\begin{equation}
\label{eq2.13} Y^{n}_{t}=E(\xi^{n}+\intt
f_{n}(s,Y^{n}_{s})\,ds+\intt dV^{n}_{s}|\FF_{t}),\quad t\in[0,T].
\end{equation}
Write $\delta Y=Y^{m}-Y^{n}$, $\delta M=M^{m}-M^{n}$, $\delta
\xi=\xi^{m}-\xi^{n}$, $\delta V=V^{m}-V^{n}$ for $m\ge n$. By the
It\^o--Tanaka formula and (H2),
\begin{align*}
|\delta Y_{t}|&\le |\delta \xi|+\int_{t}^{T}\hat\sgn(\delta
Y_{s-})(f_{m}(s,Y^{m}_{s})-f_{n}(s,Y^{n}_{s}))\,ds\\
&\quad+ \int_{t}^{T}\hat\sgn(\delta Y_{s-})d\delta
V_{s}+\int_{t}^{T}\hat\sgn(\delta Y_{s-})\,d\delta M_{s}\\
& \le|\delta \xi|+\into|f_{m}(s,Y^{n}_{s})-f_{n}(s,Y^{n}_{s})|\,ds
+\into d|\delta V|_{s}+\int_{t}^{T}\hat\sgn(\delta
Y_{s-})\,d\delta M_{s}.
\end{align*}
Conditioning both sides of the above inequality with respect to
$\FF_{t}$ and using the definitions of $\xi^{n}, f^{n}, V^{n}$ we
get
\[
|\delta Y_{t}|\le E(\Psi^n|\FF_t),
\]
where
\[
\Psi^n=|\xi|\mathbf{1}_{\{|\xi|>n\}}+\into
|f(t,0)|\mathbf{1}_{\{|f(t,0)|>n\}}\,dt+\into
\mathbf{1}_{\{|V|_{t}>n\}}\,d|V|_{t}\,.
\]
From the above one can deduce that
\[
\|\delta Y\|_{1}\le E\Psi^n
\]
and, using \cite[Lemma 6.1]{BDHPS} (see also \cite[Proposition
IV.4.7]{RY}), that
\[
E\sup_{0\le t\le T}|\delta Y_{t}|^{q}\le \frac{1}{1-q}E(\Psi^n)^q
\]
for every $q\in (0,1)$. Therefore there exists $Y\in \DM^{q}$,
$q\in (0,1)$, such that $Y$ is of class (D) and $Y^{n}\rightarrow
Y$ in the norm $\|\cdot\|_{1}$ and in $\DM^{q}$ for $q\in (0,1)$.
From the last convergence and (H1), (H3) we conclude that
\[
\into |f_{n}(t,Y^{n}_{t})-f(t,Y_{t})|\,dt\rightarrow 0.
\]
By Lemmas \ref{lm0.2} and \ref{lm0.3},
\[
\sup_{n\ge 1} E\left( \into
|f_{n}(t,Y^{n}_{t})|\,dt\right)^{q}<\infty
\]
for every $q\in(0,1)$. Therefore applying once again \cite[Lemma
6.1]{BDHPS} and letting $n\rightarrow\infty$ in (\ref{eq2.13}) we
see that $Y$ satisfies (\ref{eq2.11}) and hence the pair $(Y,M)$,
where $M$ is defined by (\ref{eq2.12}), is a solution of
BSDE$(\xi,f+dV)$. The integrability properties of $M$ follow from
Lemma \ref{lm0.2}, Lemma \ref{lm0.3} and (\ref{eq2.12}).
\end{dow}

\nsubsection{Generalized BSDEs with random terminal time}
\label{sec3}

In this section $\zeta\in\TT$, $V\in\VV$ and
$f:\mathbb{R}_{+}\times\Omega\times\mathbb{R}\rightarrow
\mathbb{R}$ is a function such that $f(\cdot,y)$ is progressively
measurable for every $y\in\mathbb{R}$.

\begin{df}
We say that a pair $(Y,M)$ is a solution of BSDE$(\zeta,f+dV)$ if
\begin{enumerate}
\item [(a)] $Y\in\DM$, $Y_{t\wedge \zeta}\rightarrow 0$
as $t\rightarrow\infty$ and $M\in\MM_{loc}$\,,
\item [(b)]for every $T>0$, $t\mapsto f(t,Y_{t})\in L^{1}(0,T)$ and
\begin{equation}
\label{eq3.1} Y_{t}=Y_{T\wedge\zeta}+\intz f(s,Y_{s})\,ds +\intz
dV_{s}-\intz dM_{s},\quad t\in [0,T].
\end{equation}
\end{enumerate}
\end{df}

Let us observe that from the above definition it follows that
$Y_{t}=Y_{t\wedge\zeta}$ for every $t\ge 0$.

We first state the analogues of Proposition \ref{lm0.1} and
Corollary \ref{cor2.2}.
\begin{stw}
\label{th1.1} Let $(Y^{1},M^{1}), (Y^{2},M^{2})$ be solutions of
BSDE$(\zeta,f^{1}+dV^{1})$ and BSDE$(\zeta,f^{2}+dV^{2})$,
respectively, such that $Y^{1}, Y^{2}$ are of class \mbox{\rm(D)}.
If $dV^{1}\le dV^{2}$ and either \mbox{\rm(\ref{eq2.01})} or
\mbox{\rm(\ref{eq2.02})} is satisfied then $Y^{1}_{t}\le
Y^{2}_{t}$, $t\ge 0$.
\end{stw}
\begin{dow}
Assume that (\ref{eq2.01}) is satisfied. Let $Y=Y^{1}-Y^{2},\,
M=M^{1}-M^{2}$ and let $\tau\in\mathcal{T}$. By the It\^o--Tanaka
formula and (H2), for every $T>0$ we have
\begin{align*}
Y^{+}_{t}&\le Y^{+}_{T\wedge\tau\wedge\zeta}
+\intzt \mathbf{1}_{\{Y^{1}_{s-}>Y^{2}_{s-}\}}(f(s,Y^{1}_{s})-f(s,Y^{2}_{s}))\,ds\\
&\quad+\intzt\mathbf{1}_{\{Y^{1}_{s-}>Y^{2}_{s-}\}}\,dM_{s} \le
Y^{+}_{T\wedge\tau\wedge\zeta}+\intzt
\mathbf{1}_{\{Y^{1}_{s-}>Y^{2}_{s-}\}}\,dM_{s},\quad t\ge 0.
\end{align*}
Let $\{\tau_{k}\}$ be a fundamental sequence for $M$. Since $Y$ is
of class (D), taking expectation of both sides of the above
inequality with $\tau$ replaced by $\tau_k$ and then letting
$k\rightarrow\infty$ we see that $EY^{+}_{t}\le
EY^{+}_{T\wedge\zeta}$, $t\ge 0$. Therefore letting
$T\rightarrow\infty$ and using once again the fact that $Y$ is of
class (D) we conclude that $Y_{t}=0$, $t\ge 0$.  In case
(\ref{eq2.02}) is satisfied the proof is analogous.
\end{dow}

\begin{wn}
\label{wn3.2}
Assume \mbox{\rm(H2)}. Then there exists at most one
solution $(Y,M)$ of BSDE$(\zeta,f+dV)$ such that $Y$ is of class
\mbox{\rm(D)}.
\end{wn}

\begin{lm}
\label{lm1.1} Let $(Y,M)$ be a solution of BSDE$(\xi,f+dV)$ on
$[0,T]$ such that $Y$ is of class \mbox{\rm(D)}. Assume
additionally that $\xi$ is $\FF_{\tau}$-measurable for some
$\tau\in\TT_{T}$, $f(\cdot,y)=0$ on the interval $(\tau,T]$ and
$\int_{\tau}^{T} d|V|_{t}=0$. Then
$(Y_{t\wedge\tau},M_{t\wedge\tau})=(Y_{t},M_{t})$, $t\in [0,T]$.
\end{lm}
\begin{dow}
Let $\{\sigma_{k}\}$ be a fundamental sequence for $M$. By the
assumptions, for every $k\in\BN$ and $\delta\in\TT_{T}$ such that
$\delta\ge\tau$,
\[
Y_{\delta}=Y_{\delta\vee\sigma_k}-\int_{\delta}^{\delta\vee\sigma_k}\,dM_{s}.
\]
Since $Y$ is of class (D) and $\xi$ is $\FF_{\tau}$-measurable, it
follows that
\begin{equation}
\label{eq1.1}
Y_{\delta}=E(Y_{\delta\vee\sigma_{k}}|\FF_{\delta})
\rightarrow E(\xi|\FF_{\delta})=\xi.
\end{equation}
By It\^o's formula,
\[
|Y_{\tau}|^{2}+\int_{\tau}^{t\vee\tau}\,d[M]_{s}
=|Y_{t\vee\tau}|^{2}-2\int_{\tau}^{t\vee\tau}Y_{s-}\,dM_{s}.
\]
By the above and (\ref{eq1.1}),
\[
\int_{\tau}^{t\vee\tau}\,d[M]_{s} =-2\int_{\tau}^{t\vee\tau}
Y_{s-}\,dM_{s},\quad t\in [0,T],
\]
which implies that $M_{t\wedge\tau}=M_{\tau},\, t\in [0,T]$. Since
$Y_{t}=\xi-\int_{t}^{T}\,dM_{s}$ for $t\in [\tau,T]$, we get the
desired result.
\end{dow}
\medskip

We can now prove our main result on existence and uniqueness of
solutions of (\ref{eq3.1}).

\begin{tw}
\label{th1.2} Assume that
$E\int_{0}^{\zeta}d|V|_{t}+E\int_{0}^{\zeta}|f(t,0)|\,dt<\infty$
and that $f$ satisfies \mbox{\rm(H1)--(H3)} for every $T>0$. Then
there exists a unique solution $(Y,M)$ of BSDE$(\zeta,f+dV)$ such
that $(Y,M)\in\DM^{q}\otimes\MM^{q}$ for $q\in (0,1)$, $M$ is a
uniformly integrable martingale and $Y$ is of class \mbox{\rm(D)}.
\end{tw}
\begin{dow}
By Theorem \ref{th0.1}, for each  $n\in\mathbb{N}$ there exists a
unique solution $(Y^{n},M^{n})$ of the BSDE
\begin{equation}
\label{eq3.07}
Y^{n}_{t}=\int_{t}^{n}\mathbf{1}_{[0,\zeta]}(s)f(s,Y^{n}_{s})\,ds
+\int_{t}^{n}\mathbf{1}_{[0,\zeta]}(s)\,dV_{s}-\int_{t}^{n}dM^{n}_{s},
\quad t\in [0,n]
\end{equation}
such that $(Y^n,M^n)\in\DM^{q}\otimes\MM^{q}$, $q\in (0,1)$, $M^n$
is uniformly integrable martingale and $Y^n$ is of class (D). Let
us put $(Y^{n}_{t},M^{n}_{t})=(0,M^{n}_{n})$ for $t\ge n$. Then by
Lemma \ref{lm1.1},
\begin{equation}
\label{T1}
(Y^{n}_{t},M^{n}_{t})=(Y^{n}_{t\wedge\zeta},M^{n}_{t\wedge\zeta}),\quad
t\ge0.
\end{equation}
For $m\ge n$ put $\delta Y=Y^{m}-Y^{n}$, $\delta M=M^{m}-M^{n}$
and
\begin{align*}
\varphi(t)&=\intot\mathbf{1}_{[0,\zeta\wedge m]}(s)f(s,Y^{m}_{s})\,ds
+\intot\mathbf{1}_{[0,\zeta\wedge m]}(s)\,dV_{s}\\&\quad
-\intot\mathbf{1}_{[0,\zeta\wedge n]}(s)f(s,Y^{n}_{s})\,ds
-\intot\mathbf{1}_{[0,\zeta\wedge n]}(s)\,dV_{s}.
\end{align*}
Then by the It\^o-Tanaka formula,
\[
|\delta Y_{t}|\le \int_{t}^{m}\hat\sgn(\delta Y_{s-})\,d\varphi(s)
-\int_{t}^{m}\hat\sgn(\delta Y_{s-})\,d\delta M_{s},\quad t\in
[0,m].
\]
Conditioning both sides of the above inequality with respect to
$\FF_{t}$ we get
\begin{equation}
\label{eq1.2} |\delta Y_{t}|\le E(\int_{t}^{m}\hat\sgn(\delta
Y_{s-}))\,d\varphi(s)|\FF_{t}),\quad t\in [0,m].
\end{equation}
Since for every $t\in [n,m]$,
\begin{equation}
\label{eq1.3} \int_{t}^{m}\hat\sgn(\delta Y_{s-})\,d\varphi(s)
=\int_{t\wedge\zeta}^{m\wedge\zeta}
\hat\sgn(Y^{m}_{s-})f(s,Y^{m}_{s})\,ds
+\int_{t\wedge\zeta}^{m\wedge\zeta}\hat\sgn(Y^{m}_{s-})\,dV_{s},
\end{equation}
using \cite[Lemma 6.1]{BDHPS} and (H2) we deduce from
(\ref{eq1.2}) that for every $q\in(0,1)$,
\begin{equation}
\label{eq1.4} E\sup_{n\le t\le m}|\delta Y_{t}|^{q}
\le\frac{1}{1-q}E(\int_{n\wedge\zeta}^{\zeta}|f(s,0)|\,ds
+\int_{n\wedge\zeta}^{\zeta}d|V|_{s})^{q}.
\end{equation}
Observe that for $t\in[0,n]$,
\begin{equation}
\label{eq1.5} \int_{t}^{m}\hat\sgn(\delta
Y_{s-})\,d\varphi(s)=\int_{n}^{m}\hat\sgn(\delta
Y_{s-})\,d\varphi(s)+\int_{t}^{n}\hat\sgn(\delta
Y_{s-})\,d\varphi(s)
\end{equation}
and
\begin{equation}
\label{eq1.6}
\varphi(t)=\intot \mathbf{1}_{[0,\zeta]}(s)
(f(s,Y^{m}_{s})-f(s,Y^{n}_{s}))\,ds.
\end{equation}
From (\ref{eq1.2}), (\ref{eq1.5}), (\ref{eq1.6}) and \cite[Lemma
6.1]{BDHPS} it follows that for every $q\in (0,1)$,
\begin{equation}
\label{eq1.7} E\sup_{0\le t\le n}|\delta Y_{t}|^{q}
\le\frac{1}{1-q}E(\int_{n\wedge\zeta}^{\zeta}|f(s,0)|\,ds
+\int_{n\wedge\zeta}^{\zeta}d|V|_{s})^{q}.
\end{equation}
Combining (\ref{eq1.4}) with  (\ref{eq1.7}) we see that for every
$q\in(0,1)$,
\begin{equation}
\label{eq1.8} E\sup_{t\ge 0}|\delta Y_{t}|^{q}\le
\frac{1}{1-q}E(\int_{n\wedge\zeta}^{\zeta}|f(s,0)|\,ds
+\int_{n\wedge\zeta}^{\zeta}d|V|_{s})^{q}.
\end{equation}
Using once again (\ref{eq1.3}), (\ref{eq1.5}), (\ref{eq1.6}) we
deduce from (\ref{eq1.2}) that
\[
\|\delta Y\|_{1}\le
\frac{1}{1-q}E(\int_{n\wedge\zeta}^{\zeta}|f(s,0)|\,ds
+\int_{n\wedge\zeta}^{\zeta} d|V|_{s})^{q}.
\]
Therefore there exists $Y\in\DM^{q}$, $q\in (0,1)$, such that $Y$
is of class (D), $Y^{n}\rightarrow Y$ in the norm $\|\cdot\|_{1}$
and in $\DM^{q}$ for  $q\in (0,1)$. By the latter convergence and
(\ref{T1}), $Y_{t\wedge\zeta}\rightarrow 0$ as $t\rightarrow
\infty$. By Lemma \ref{lm0.3}, for any $m\ge n\ge T>0$,
\[
E[\delta M]_{T}^{q/2} \le c_{q}E\sup_{0\le t\le T}|\delta
Y_{t}|^{q},\quad q\in (0,1).
\]
From this and (\ref{eq1.8}) it follows  that there exists
$M\in\MM$ such that for every $q\in(0,1)$ and $T>0$,
\begin{equation}
\label{eq1.10} E[M^{n}-M]_{T}^{q/2}\rightarrow0
\end{equation}
as $n\rightarrow\infty$. By (H1), (H3), (\ref{eq1.8}) and the
Lebesgue dominated convergence theorem,
\begin{equation}
\label{eq1.11} \into|f(s,Y^{n}_{s})-f(s,Y_{s})|\,ds\rightarrow 0
\end{equation}
as $n\rightarrow\infty$. By the definition of the processes
$(Y^{n},M^{n})$, for every $T>0$,
\[
Y^{n}_{t}=Y^{n}_{T\wedge\zeta}+\intz f(s,Y^{n}_{s})\,ds
+\intz dV_{s}-\intz dM^{n}_{s},\quad t\ge 0.
\]
Therefore letting $n\rightarrow\infty$ and using
(\ref{eq1.8})--(\ref{eq1.11}) we see that $(Y,M)$ satisfies
(\ref{eq3.1}). What is left is to show integrability properties of
$M$.  That $M\in\MM^{q}$, $q\in (0,1)$, follows from the fact that
$Y\in\DM^{q}$ for $q\in (0,1)$ and Lemma \ref{lm0.3}. By Lemmas
\ref{lm0.2} and \ref{lm0.3},
$E\int_{0}^{\zeta}|f(s,Y_{s})|\,ds<\infty$. Using this, the fact
that $Y$ is of class (D) and $Y_{t\wedge\zeta}\rightarrow0$ as
$t\rightarrow\infty$ it is easy to deduce from (\ref{eq3.1}) that
$M$ has the form
\[
M_{t}=E(\int_{0}^{\zeta} f(s,Y_{s})\,ds
+\int_{0}^{\zeta} dV_{s}|\FF_{t})-Y_{0},\quad t\ge 0.
\]
Thus, $M$ is closed and hence uniformly integrable.
\end{dow}
\medskip

Let $E$ be a Radon space (see \cite{Sharpe}) and let
$\BX=(\Omega,\FF,\FF_t,X,\theta_t,\zeta,P_{x},)$ be a right
process (with translation operators $\theta_t$ and life-time
$\zeta$) on $E$. Suppose we are given a measurable function
$f:E\times\mathbb{R}\rightarrow\BR$ and a finite variation
additive functional $V$ of $\BX$. Then for $x\in E$, $r\ge 0$ we
put $\zeta^{r}=\zeta+r, V^{r}=V_{\cdot-r}$ and we define
$f^{r}:[r,\infty)\times\Omega\times\BR\rightarrow\BR$ by putting
$f^{r}(t,y)(\omega)=f(X_{t-r}(\omega),y)$.

\begin{stw}
\label{prop2.1} Assume that for every $r\ge0$ the function $f^{r}$
satisfies \mbox{\rm(H2)} and that there exists a solution
$(Y^{r,x},M^{r,x})=(Y^{r},M^{r})$ of BSDE$(\zeta^{r},
f^{r}+dV^{r})$ on $[r,\infty)$, defined on the space
$(\Omega,\FF,\FF_{\cdot-r},P_{x})$, such that $Y^{r}$ is of class
\mbox{\rm(D)}. Then for every $h\ge 0$,
\begin{enumerate}
\item[\rm(i)]$(M^{r+h}_{t}\circ\theta_{h},\FF_{t-r}, t\ge r+h)$ is a
local martingale under $P_{x}$,
\item[\rm(ii)]$(Y^{r+h}_{t}\circ\theta_{h}, M^{r+h}_{t}\circ\theta_{h})
=(Y^{r}_{t},M^{r}_{t})$, $t\ge r+h$, $P_{x}$-\mbox{\rm a.s.},
\item[\rm(iii)] $(Y^{r+h}_{t+h},M^{r+h}_{t+h})=(Y^{r}_{t},M^{r}_{t})$,
$t\ge r\ge 0$, $h\ge 0$, $P_{x}$-\mbox{\rm a.s.}.
\end{enumerate}
\end{stw}
\begin{dow}
(i) Let $N_{t}=M^{r+h}_{t+r+h}$. By the assumption,
$(N_{t},\FF_{t}, t\ge 0)$ is a local martingale. Hence, by
\cite[Proposition 50.19]{Sharpe},
$(N_{t-h}\circ\theta_{h}\mathbf{1}_{[h,+\infty)}(t),\FF_{t},t\ge0)$
is again a local martingale. But
$N_{t-h}\circ\theta_{h}=M^{r+h}_{t+r}\circ\theta_{h},\, t\ge h$
which implies (i).

(ii) By the assumption,
\[
Y^{r}_{t}=Y^{r}_{T}+\int_{t\wedge\zeta^{r}}^{T\wedge\zeta^{r}}
f(X_{s-r},Y^{r}_{s})\,ds
+\int_{t\wedge\zeta^{r}}^{T\wedge\zeta^{r}}dV^{r}_{s}
-\int_{t\wedge\zeta^{r}}^{T\wedge\zeta^{r}}dM^{r}_{s},\quad
t\in [r,T]
\]
and
\begin{align*}
Y^{r+h}_{t}&=Y^{r+h}_{T}+\int_{t\wedge\zeta^{r+h}}^{T\wedge\zeta^{r+h}}
f(X_{s-r-h},Y^{r+h}_{s})\,ds\\
&\quad+\int_{t\wedge\zeta^{r+h}}^{T\wedge\zeta^{r+h}} dV^{r+h}_{s}
-\int_{t\wedge\zeta^{r+h}}^{T\wedge\zeta^{r+h}}dM^{r+h}_{s},\quad
t\in [r+h,T].
\end{align*}
Hence
\begin{align*}
Y^{r+h}_{t}\circ\theta_{h}&=Y^{r+h}_{T}\circ\theta_{h}
+\int_{t\wedge\zeta^{r,h}}^{T\wedge\zeta^{r,h}}
f(X_{s-r},Y^{r+h}_{s}\circ\theta_{h})\,ds\\
&\quad+\int_{t\wedge\zeta^{r,h}}^{T\wedge\zeta^{r,h}}
d(V^{r+h}_{s}\circ\theta_{h})
-\int_{t\wedge\zeta^{r,h}}^{T\wedge\zeta^{r,h}}
d(M^{r+h}_{s}\circ\theta_{h}),\quad t\in [r+h,T],
\end{align*}
where $\zeta^{r,h}=(\zeta-h)^{+}+r+h$. By part (i),
$(M^{r+h}_{t}\circ\theta_{h},\FF_{t-r}, t\ge r+h)$ is a local
martingale. Since $V$ is additive,
$d(V^{r+h}_{s}\circ\theta_{h})=d V^{r}_{s}$. Observe also that if
$\zeta\ge h$ then $\zeta^{r,h}=\zeta^{r}$ and if $\zeta\le h$ then
$T\wedge\zeta^{r}=t\wedge\zeta^{r},
T\wedge\zeta^{h,r}=t\wedge\zeta^{h,r}$ for $t\in [r+h,T]$.
Therefore,
\begin{align*}
Y^{r+h}_{t}\circ\theta_{h}&=Y^{r+h}_{T}\circ\theta_{h}
+\int_{t\wedge\zeta^{r}}^{T\wedge\zeta^{r}}
f(X_{s-r},Y^{r+h}_{s}\circ\theta_{h})\,ds\\
&\quad+\int_{t\wedge\zeta^{r}}^{T\wedge\zeta^{r}} dV^{r}_{s}
-\int_{t\wedge\zeta^{r}}^{T\wedge\zeta^{r}}d(M^{r+h}_{s}\circ\theta_{h}),
\quad t\in [r+h,T].
\end{align*}
We see that $(Y^{r},M^{r})$, $(Y^{r+h}\circ\theta_{h},
M^{r+h}\circ\theta_{h})$ are solutions of
BSDE$(\zeta^{r},f^{r}+dV^{r})$ on $[r+h,+\infty)$ defined on the
space $(\Omega,\FF,\FF_{\cdot-r},P_{x})$. Therefore (ii) follows
from Corollary \ref{wn3.2}. Since the proof of (iii) is analogous
to that of (ii), we omit it.
\end{dow}

\begin{uw}
\label{uw2.1} Let $B$ be a Borel subset of $E$ and for $x\in B$
let the pair $(Y^x,M^x)$ be a unique solution of BSDE$(\xi,f+dV)$
of Theorem \ref{th1.2} defined on the filtered probability space
$(\Omega,\FF,\FF_t,P_x)$. Then there exists a pair $(Y,M)$ of
$(\FF_t)$ adapted c\`adl\`ag processes such that $(Y_{t},M_{t})
=(Y^{x}_{t},M^{x}_{t})$, $t\ge 0$, $P_{x}$-a.s. for every $x\in
B$. This follows from the construction of solutions $(Y^x,M^x)$
and repeated application of Lemmas A.3.3 and A.3.5 in
\cite{Fukushima}. Indeed, let $(Y^{x,n},M^{x,n})$ be a solution of
(\ref{eq3.07}) on $(\Omega,\FF,\FF_t,P_x)$. Since
$Y^{x,n}\rightarrow Y^x$ in probability $P_x$ for $x\in B$, to
prove the desired result it suffices to show that there exists a
pair $(Y^n,M^n)$ of $(\FF_t)$ adapted c\`adl\`ag processes such
that $(Y^n_{t},M^n_{t})=(Y^{x,n}_{t},M^{x,n}_{t})$, $t\ge 0$,
$P_{x}$-a.s. for every $x\in B$. But the solution of
(\ref{eq3.07}) is a limit in probability of solutions of equations
considered in Lemma \ref{lm0.4} (see the proof of Theorem
\ref{th0.1}), and solutions of equations considered in Lemma
\ref{lm0.4} are limits of Picard iterations of solutions of linear
equations of the form (\ref{eq0.2}). Using \cite[Lemma
A.3.5]{Fukushima} one can find independent of $x$ solutions of
these linear equations. Consequently, using \cite[Lemma
A.3.3]{Fukushima} we can find $(Y^n,M^n)$ having the desired
properties.
\end{uw}

\nsubsection{Probabilistic solutions of equations with measure
data} \label{sec4}

In the rest of the paper we assume that
\begin{itemize}
\item $E$
is a locally compact separable metric space and $m$ is a positive
Radon measure on $E$ such that $\mbox{supp}[m]=E$, i.e. $m$ is a
nonnegative Borel measure on $E$ finite on compact sets and
strictly positive on nonempty open sets,
\item
$(\mathcal{E},D[\EE])$ is a regular Dirichlet form on
$L^{2}(E;m)$.
\end{itemize}

Let $D[\EE]$ be a dense linear subspace of $L^2(E;m)$ and let
$\EE$ be a nonnegative symmetric bilinear form on $D[\EE]\times
D[\EE]$. Set $\EE_{\alpha}(u,v)=\EE(u,v)+\alpha(u,v)$ for $u,v\in
D[\EE]$, $\alpha>0$.

Let us recall that $(\EE,D[\EE])$ is called a Dirichlet form if it
is closed, i.e. $D[\EE]$ is complete under the norm $\EE_{1}$, and
Markovian,  i.e. if $u\in D[\EE]$ and $v$ is a normal contraction
of $u$ then $v\in D[\EE]$ and $\EE(u,u)\le\EE(v,v)$ (a function
$v$ is called a normal contraction of $u$ if
$|v(x)-v(y)|\le|u(x)-u(y)|$ and $|v(x)|\le |u(x)|$ for $x,y\in
E$).

A Dirichlet form $(\EE,D[\EE])$ is called regular if the space
$D[\EE]\cap C_{0}(E)$ is dense in $D[\EE]$ with respect to the
norm $\EE_{1}$ and dense in $C_{0}(E)$ with respect to the uniform
convergence topology, where $C_{0}(E)$ is the space of continuous
functions on $E$ with compact support.

Let $\mbox{cap}:2^{E}\rightarrow \BR^{+}$ denote the Choquet
capacity associated with the  form $(\EE,D[\EE])$ (see
\cite[Chapter 2]{Fukushima}). In the sequel we  say that a
statement depending on $x\in E$ holds quasi-everywhere (``q.e.''
for short) on $E$ if there is a set $B\subset E$ of capacity zero
such that the statement is true for every $x\in B$.

A function $u:E\rightarrow\BR$ is called quasi-continuous if for
every $\varepsilon>0$ there exists an open set $U\subset E$ such
that $\mbox{cap}(U)<\varepsilon$ and $u_{|E\setminus U}$ is
continuous. It is known that each $u\in D[\EE]$ admits a
quasi-continuous $m$-version (see \cite[Theorem
2.1.3]{Fukushima}). In the sequel we always consider a
quasi-continuous version of $u$ if it has such a version.

Let $\BX=(\Omega,\FF,\FF_t,X,\theta_t,P_{x})$ be a (unique) Hunt
process associated with $(\EE,D[\EE])$ (see Chapter 7 and Appendix
A.2 in \cite{Fukushima}). In what follows by $\zeta$ we denote the
life-time of $X$, i.e. $\zeta=\inf\{t\ge 0; X_{t}=\Delta\}$, where
$\Delta$ is the one-point compactification of $E$. If $E$ is
already compact, $\Delta$ is adjoint as an isolated point.

For $B\subset E$ we set
\[
\sigma_{B}(\omega)=\inf\{t>0; X_{t}(\omega)\in B\}.
\]

A set $B\subset E$ is called nearly Borel if for each finite
nonnegative Borel measure $\nu$ on $E$ there exist Borel sets
$B_{1}, B_{2}$ such that $B_{1}\subset B\subset B_{2}$ and
$P_{\nu}(\exists\, t\ge 0;\, X_{t}\in B_{2}\setminus B_{1})=0$,
where $P_{\nu}(\cdot)=\int P_x(\cdot)\,\nu(dx)$. A set $N\subset
E$ is called exceptional if there exists a nearly Borel set
$\tilde{N}$ such that $N\subset \tilde{N}$ and
$P_{m}(\sigma_{\tilde{N}}<\infty)=0$. By \cite[Theorem
4.2.1]{Fukushima}, a set $N\subset E$ is exceptional iff
$\mbox{cap}(N)=0$.

Let $\BB(E)$ ($\mathcal{B}^{n}(E)$) denote the space of all Borel
(nearly Borel) measurable functions $u:E\rightarrow\BR$ and let
$\mathcal{C}$ denote the space of all $u\in\BB^{n}(E)$ for which
there exists an exceptional Borel set $B\subset E$ such that the
process $t\rightarrow u(X_{t})$ is right continuous and
$t\rightarrow u(X_{t-})$ is left continuous on $[0,\zeta)$ under
$P_{x}$ for every $x\in B^{c}$. By \cite[Theorem 4.2.2]{Fukushima}
and \cite[Theorem 5.29]{MR}, $u\in\BB^{n}(E)$ is quasi-continuous
iff it belongs to $\mathcal{C}$.

An increasing sequence $\{F_{n}\}$ of closed subsets of $E$ is
called a generalized nest if
\begin{equation}
\label{eq4.01} \lim_{n\rightarrow\infty}\mbox{cap}(K\setminus
F_{n})=0
\end{equation}
for any compact set $K\subset E$. $\{F_n\}$ is called a nest if
(\ref{eq4.01}) holds with $E$ in place of $K$.

Recall that a Borel measure $\mu$ on $E$ is called smooth if its
total variation $|\mu|$ charges no set of zero capacity and there
exists a generalized nest $\{F_{n}\}$ such that
$|\mu|(F_{n})<\infty$ for every $n\in\mathbb{N}$. The set of all
nonnegative smooth measures on $E$ will be denoted by $S$.

It is known (see \cite[Chapter 5]{Fukushima}) that for every
measure $\mu\in S$ there exists a unique positive continuous
additive functional (PCAF) $A$ of $\BX$ which is in the Revuz
correspondence with  $\mu$, i.e. for every bounded nonnegative
$f\in\BB(E)$,
\begin{equation}
\label{eq4.2} \lim_{t\searrow 0}
\frac{1}{t}E_{m}\int_{0}^{t}f(X_{s})\,dA_{s}
=\int_{E}f(x)\,\mu(dx).
\end{equation}
Moreover, any PCAF $A$ of $\BX$ admits a unique measure $\mu\in
S$, which is called the Revuz measure of $A$, such that
(\ref{eq4.2}) is satisfied. Thus, the Revuz correspondence
(\ref{eq4.2}) provides probabilistic description of $S$.

From the analytic description of $S$ it is easy to see that $S$
contains all positive Radon measures on $E$ charging no set of
zero capacity.  The following example shows that in general the
inclusion is strict.

\begin{prz}
(see \cite[Example 5.1.1]{Fukushima}) Let $E=\BR^d$ and let $m$ be
the $d$-dimensional Lebesgue measure. Consider the form
$\EE(u,v)=\frac12\sum^d_{i=1}\int_{\BR^d}\frac{\partial
u}{\partial x_i}\frac{\partial v}{\partial x_i}\,dx$,
$D[\EE]=H^1(\BR^d)$ on $L^2(\BR^d;dx)$. Then the process $\BX$
associated with $(\EE,D[\EE])$ is a standard Brownian motion on
$\BR^d$. \\
(i) Assume that $d\ge2$ and let $\mu(dx)=g(x)\,dx$, where
\[
g(x)=|x|^{-\alpha},\quad x\in \BR^d
\]
for some $\alpha\ge d$. Then  $\mu$ is smooth but not Radon. The
PCAF of $\BX$ corresponding to $\mu$ has the form
\[
A_t=\int^t_0g(X_s)\,ds,\quad t\ge0.
\]
(ii) If $d=1$ then $\mu\in S$ iff it is a positive Radon measure.
The corresponding PCAF is given by
\[
A_t=\int_{\BR}L^a_t\,\mu(da),\quad t\ge0,
\]
where $\{L^a_t,t\ge0,a\in\BR\}$ denotes the continuous (in the
variables $t$ and $a$) version of the local time of $X$.
\end{prz}

Let us also note that one can construct smooth measures that are
``nowhere Radon" in the sense that $\mu(U)=\infty$ for every
nonempty open set $U\subset E$. Many interesting examples of such
measures are to be find in \cite{AM}.

In the sequel, given a nonnegative Borel measure $\mu$ on $E$ and
$f\in\BB^{+}(E)$ we put
\[
\int_{E} f\,d\mu=\langle f,\mu\rangle.
\]
By $f\cdot\mu$ we denote the Borel measure such that
$\frac{df\cdot\mu}{d\mu}=f$.

\begin{lm}
\label{lm2.1} Let $A$ be a PCAF of $\BX$. Then for every stopping
time $\tau$, $E_{x}A_{\tau}$ is finite for $m$-a.e. $x\in E$ iff
$E_{x}A_{\tau}$ is finite for q.e. $x\in E$.
\end{lm}
\begin{dow}
Sufficiency follows from the definition of the capacity. To prove
necessity let us assume that $E_{x}A_{\tau}<\infty$ for  $m$-a.e.
$x\in E$ and set $B=\{x\in E; E_{x}A_{\tau}<\infty\}$. Since $\BX$
is a Hunt process, $B$ is a nearly Borel set. Let $K$ be a compact
set such that $K\subset B$. Since $\BX$ is strong Markov and  $A$
is additive, for $m$-a.e. $x\in E$ we have
\begin{align*}
P_{x}(\sigma_{K}<\infty)
&=P_{x}(E_{X_{\sigma_{K}}}A_{\tau}=\infty)
=P_{x}(E_{x}(A_{\tau}\circ\theta_{\sigma_{K}}|\FF_{\sigma_{K}})
=\infty)\\
&=P_{x}(E_{x}(A_{\tau}-A_{\sigma_{k}\wedge\tau}|\FF_{\sigma_{K}})
=\infty)=0.
\end{align*}
Thus, $P_{m}(\sigma_{K}<\infty)=0$ or, equivalently,
$\mbox{cap}(K)=0$. Since  this holds for arbitrary compact set
$K\subset B$ and $\mbox{cap}$ is a Choquet capacity,
$\mbox{cap}(B)=0$.
\end{dow}

\begin{lm}
\label{lm2.2} Assume that $A$ is PCAF of $\BX$ such that
$E_{x}A_{\zeta}<\infty$ for $m$-a.e. $x\in E$. Then the function
\[
u(x)=E_{x}A_{\zeta},\quad x\in E
\]
is quasi-continuous.
\end{lm}
\begin{dow}
By Lemma \ref{lm2.1}, $u(x)<\infty$ q.e. Hence, by \cite[Theorem
4.1.1]{Fukushima}, without loss of generality we may assume that
$B=\{x\in E; u(x)=\infty\}$ is properly exceptional. Since $A$ is
PCAF of $\BX$, by \cite[Theorem 5.1.4]{Fukushima} there exists a
unique measure $\mu\in S$ such that $A=A^{\mu}$. Let $\{F_{n}\}$
be a generalized nest such that $\mathbf{1}_{F_{n}}\mu\in S_{00}$
(see \cite[Theorem 2.2.4]{Fukushima}). Then for each $n\in\BN$ the
function $u_n$ defined by
\[
u_{n}(x)\equiv E_{x}\int_{0}^{\infty}
e^{-t/n}\,\mathbf{1}_{F_{n}}(X_{t})\,dA_{t}
=E_{x}\int_{0}^{\infty}
e^{-t/n}\,dA_t^{\mathbf{1}_{F_{n}}\mu}<\infty,\quad x\in E
\]
is quasi-continuous (see Theorems 5.1.1 and 5.1.6 in
\cite{Fukushima}). Let us observe that
\begin{equation}
\label{eq2.1}
u_{n}(x)\nearrow u(x),\quad x\in B^{c}.
\end{equation}
Indeed, since $E$ is locally compact separable metric space, to
prove (\ref{eq2.1}) it suffices to show that for every compact set
$K\subset E$,
\begin{equation}
\label{eq2.2}
\lim_{n\rightarrow\infty}E_{x}\int_{0}^{\infty}\mathbf{1}_{K\setminus
F_{n}}(X_{t})\,dA_{t}=0.
\end{equation}
By \cite[Theorem 4.2.1]{Fukushima}, $p^{1}_{K\setminus
F_{n}}(x)\rightarrow 0$, q.e., where $p^{1}_{K\setminus F_{n}}(x)
=E_{x}e^{-\sigma_{K\setminus F_{n}}},\, x\in E$. In view of
\cite[Theorem 4.1.1]{Fukushima}, without loss of generality we may
assume that $p^{1}_{K\setminus F_{n}}(x)\rightarrow 0$ for every
$x\in B^{c}$. The last convergence implies that for every $x\in
B^{c}$,
\begin{equation}
\label{eq2.3} \mathbf{1}_{K\setminus F_{n}}(X_{t})\rightarrow
0,\quad t\ge 0,\quad P_{x}\mbox{-a.s.}.
\end{equation}
Using this, the definition of $B$ and the Lebesgue dominated
convergence theorem we get (\ref{eq2.2}), and hence (\ref{eq2.1}).
From (\ref{eq2.3}) it also follows that for every $x\in B^{c}$,
\begin{equation}
\label{eq2.4}
\lim_{n,m\rightarrow\infty}E_{x}\int_{0}^{\infty}\mathbf{1}_{F_{n}\Delta
F_{m}}(X_{t})\,dA_{t}=0.
\end{equation}
By the Markov property and \cite[Lemma 6.1]{BDHPS},
\[
E_{x}\sup_{t\ge 0}|u_{n}(X_{t})-u_{m}(X_{t})|^{q} \le
\frac{1}{1-q}E_{x}(\int_{0}^{\infty} \mathbf{1}_{F_{n}\Delta
F_{m}}(X_{t})\,dA_{t})^{q},\quad q\in (0,1).
\]
Combining this with (\ref{eq2.1}), the fact that $B$ is properly
exceptional set and (\ref{eq2.4}) shows that for every $x\in B$,
\begin{equation}
\label{eq4.6} \lim_{n\rightarrow\infty}E_{x}\sup_{t\ge
0}|u_{n}(X_{t})-u(X_{t})|^{q}=0.
\end{equation}
Since $u_{n}$ is quasi-continuous, $u_{n}\in\mathcal{C}$. From
this and (\ref{eq4.6}) it may be concluded that $u\in\mathcal{C}$,
i.e. $u$ is quasi-continuous.
\end{dow}
\medskip

Let $A$ denote the unique nonpositive self-adjoint operator on
$L^{2}(E;m)$ such that
\[
D(A)\subset D[\EE], \quad \EE(u,v)=(-Au,v),\,u\in D(A),v\in D[\EE]
\]
(see \cite[Corollary 1.3.1]{Fukushima}) and let
$\BX=(\Omega,\FF,\FF_t,X,P_x)$ be a Hunt process with life-time
$\zeta$ associated with $(\EE,D[\EE])$.


\begin{df}
Let $\mu$ be a smooth measure such that
$E_x|A^{\mu}|_{\zeta}<\infty$ for $m$-a.e. $x\in E$, where
$A^{\mu}$ is the CAF of $\BX$ associated with $\mu$. We say that a
quasi-continuous function $u:E\rightarrow\BR$ is a probabilistic
solution of the equation
\begin{equation}
\label{eq2.5} -Au=f_u+\mu,
\end{equation}
where $f_u=f(\cdot,u)$, if
$E_{x}\int_{0}^{\zeta}|f_u(X_{t})|\,dt<\infty$ and
\begin{equation}
\label{eq2.7} u(x)=E_{x}\int_{0}^{\zeta}f_u(X_{t})\,dt
+E_{x}\int_{0}^{\zeta}dA^{\mu}_{t}
\end{equation}
for q.e. $x\in E$.
\end{df}

In Section \ref{sec6}  we provide a simple example of a Dirichlet
form and $\mu\in S$ such that $\mu$ is not Radon and
$E_x|A^{\mu}|_{\zeta}<\infty$ for $m$-a.e. $x\in E$.

We now introduce an important notion of quasi-$L^{1}$ functions on
$E$ (we recall that it appears in condition (A3$'$)).

\begin{df}
We say that a Borel function $f$ on $E$ is quasi-$L^{1}$ with
respect to the regular Dirichlet form $(\EE,D(\EE))$ on
$L^{2}(E;m)$  if the function $t\mapsto f(X_{t})$ belongs to
$L^{1}_{loc}(\BR_{+})$ $P_{x}$-a.s. for q.e. $x\in E$.
\end{df}

\begin{uw}
\label{uwww} If $f\in L^{1}(E;m)$ then $f$ is quasi-$L^{1}$.
Indeed, if $f\in L^{1}(E;m)$ then by \cite[Theorem
5.1.3]{Fukushima},
\[
E_{m}\into|f(X_{t})|\,dt=\into\langle|f|,p_{t}1\rangle\,dt\le
T\|f\|_{L^{1}(E;m)},
\]
where $\{p_{t},\,t\ge 0\}$ is the semigroup associated with the
operator $A$ corresponding to $\EE$. Therefore $E_{x}\into
|f(X_{t})|\,dt<\infty$ for $m$-a.e. $x\in E$, and hence, by Lemma
\ref{lm2.2}, for q.e. $x\in E$.
\end{uw}

\begin{uw}
\label{uw4.4} A different notion of quasi-integrability was
introduced in the paper \cite{OrsinaPonce} devoted to  semilinear
elliptic systems with measure data. In \cite{OrsinaPonce}, where
the Laplace operator $\Delta$ on a smooth bounded domain
$D\subset\BRD$ subject to the Dirichlet boundary conditions is
considered, a measurable function $f:D\rightarrow\BR$ is called
quasi-$L^1$ if for every $\varepsilon>0$ and compact set $K\subset
D$ there exists an open set $U\subset D$ such that
$\mbox{cap}(U)<\varepsilon$ and $f\in L^{1}(K\setminus U;dx)$. By
\cite[Proposition 2.3]{OrsinaPonce}, $f$ is quasi-$L^{1}$ on $D$
iff there exists a quasi-finite function $G$ on $D$ and $H\in
L^{1}(D;dx)$ such that $|f|\le G+H$, $m$-a.e., where $m$ is the
Lebesgue measure on $D$ and $\mbox{cap}$ is the Newtonian
capacity, i.e. the capacity associated with the form generating
$\Delta$ (see Section \ref{sec6}). Here ``quasi-finite'' means
that for every $\varepsilon>0$ and every compact set $K\subset D$
there exists $M>0$ and an open set $U\subset D$ such that
$\mbox{cap}(U)<\varepsilon$ and $|G|\le M$, $m$-a.e. on
$K\setminus U$.

Let us observe that if $f$ is quasi-$L^1$ in the sense of
\cite{OrsinaPonce} than for every compact subset $K\subset D$, the
function $f{|_K}$ is  quasi-$L^1$ in the sense defined in our
paper. To see this, let us first note that by Remark \ref{uwww},
$H$ is  quasi-$L^{1}$ in the sense of our definition. Since $G$ is
quasi-finite, there exists a decreasing sequence $\{U_{n}\}$ of
open subsets of $D$ and a sequence $\{M_{n}\}$ of positive
constants such that $\mbox{cap}(U_{n})\searrow0$ and
$G_{|K\setminus U_{n}}\le M_{n}$, $m$-a.e. In particular, $G\in
L^{1}(K\setminus U_{n};dx)$ for $n\in\mathbb{N}$. From this and
\cite[Theorem 4.2.1]{Fukushima} it follows that for q.e. $x\in E$,
\begin{align*}
&P_{x}(\int_{0}^{T}G_{|K}(X_{t})\,dt=\infty)\\
&\qquad \le P_{x}(\into G_{|K\setminus U_{n}}(X_{t})\,dt=\infty)
+P_{x}(\into G_{|U_{n}}(X_{t})\,dt=\infty)\\
&\qquad=P_{x}(\into G_{|U_{n}}(X_{t})\,dt=\infty)\\
&\qquad \le P_{x}(\exists_{t\in [0,T]} X_{t}\in
U_{n})=P_{x}(\sigma_{U_{n}}\le T)\le
e^{T}E_{x}e^{-\sigma_{U_{n}}},
\end{align*}
which converges to zero as $n\rightarrow\infty$. Thus, $G_{|K}$ is
quasi-$L^{1}$, which completes the proof that $f_{|K}$ is
quasi-$L^{1}$.

The class of quasi-$L^1$ functions defined in \cite{OrsinaPonce}
is well adjusted to the  Dirichlet problem with zero boundary
conditions. It is, however, too large to get existence results in
our general setting (for instance if the Dirichlet form leads to
equations with the Laplace operator subject to Neumann boundary
conditions). To overcome this difficulty one can define
analytically a bit narrower class of functions, say the class
$qL^{1}$, consisting of all measurable $f:D\rightarrow\BR$ such
that for every $\varepsilon>0$ there exists an open set $U\subset
D$ such that $\mbox{cap}(U)<\varepsilon$ and $f\in
L^{1}(D\setminus U;dx)$. Then in the same manner as above (with
$K=D$) one can show that if $f\in qL^{1}$ then $f$ is
quasi-$L^{1}$ in the sense of our definition. The class $qL^{1}$
is in general narrower then the class of quasi-$L^{1}$ defined in
the present paper. To see this it suffices to consider the
Dirichlet form (\ref{eq4.0}) with $D=\BRD$ and coefficients
$a_{ij}$ satisfying (\ref{eq4.1}) and condition (b). Then every
continuous function $f$ on $\BRD$ is quasi-$L^{1}$ but $f\equiv 1$
does not belong to $qL^{1}$.
\end{uw}

\begin{uw}
The space of quasi-$L^{1}$ functions is quite wide. It contains
many singular functions. In case $A$ is as in Remark \ref{uw4.4},
a typical example of such function is
$f:B^{d}(0,1)\rightarrow\BR$, $d\ge 2$, defined by
$f(x)=|x|^{-\alpha}$ for some $\alpha>0$.
\end{uw}

In order to state succinctly our main theorem on existence and
uniqueness of solutions of (\ref{eq2.5}), we introduce the
following terminology. We say that a function $u:E\rightarrow\BR$
is of class (FD) if the process $t\mapsto u(X_t)$ is of class (D)
under the measure $P_x$ for q.e. $x\in E$. Similarly, we say that
$u\in\FD^p$ if the process $t\mapsto u(X_t)$ belongs to $\DM^p$
under $P_x$ for q.e. $x\in E$.

Let us recall that $f^0(t,y)(\omega)=f(X_t(\omega),y)$,
$t\ge0,y\in\BR$.
\begin{tw}
\label{th2.1} Assume \mbox{\rm(A1), (A2), (A3$'$), (A4$'$)}. Then
there exists a unique solution $u$ of \mbox{\rm(\ref{eq2.5})} such
that $u$ is of class \mbox{\rm(FD)}. Actually, $u\in\FD^{q}$ for
$q\in (0,1)$.  Moreover, for q.e. $x\in E$ there exists a unique
solution $(Y^{x},M^{x})$ of BSDE$(\zeta,f^{0}+dA^{\mu})$ on
$(\Omega,\FF,\FF_t,P_{x})$. In fact,
\[
u(X_{t})=Y^{x}_{t},\quad t\ge 0, \quad P_{x}\mbox{\rm-a.s.}.
\]
\end{tw}
\begin{dow}
By Lemma \ref{lm2.1}, condition (A4$'$) is satisfied q.e. Let us
denote by $N$ the set of those $x\in E$ for which (A4$'$) is not
satisfied. In view of \cite[Theorem 4.1.1]{Fukushima} we may
assume that $N$ is properly exceptional. By Theorem \ref{th1.2},
for $x\in N^{c}$ there exists a unique solution $(Y^{x},M^{x})$ of
BSDE$(\zeta, f^{0}+dA^{\mu})$ on $(\Omega,\FF,\FF_t,P_{x})$ such
that $Y^{x}\in\DM^{q}, q\in (0,1)$, $Y^{x}$ is of class (D) and
$M^{x}\in\MM^{q}$, $q\in (0,1)$. By Remark \ref{uw2.1} there
exists a pair $(Y,M)$ of $(\FF_t)$ adapted c\`adl\`ag processes
which is a version of $(Y^{x},M^{x})$ under $P_{x}$ for every
$x\in N^{c}$. Let us put
\[
u(x)=E_{x}Y_{0},\quad x\in N^{c}, \qquad u(x)=0,\quad x\in N.
\]
By the Markov property, Proposition \ref{prop2.1} and the fact
that $N$ is properly exceptional, for every $x\in N^{c}$ we have
\[
u(X_{t})=E_{X_{t}}Y_{0}=E_{X_{t}}Y^{t}_{t}
=E_{x}(Y^{t}_{t}\circ\theta_{t}|\FF_{t})
=E_{x}(Y^{0}_{t}|\FF_{t})=Y^{0}_{t},\quad P_{x}\mbox{-a.s.}.
\]
Since $u\in\mathcal{C}$, $u(X_{t})=Y^{x}_{t}$, $t\ge 0$,
$P_{x}$-a.s. for  q.e. $x\in E$, and the proof is complete.
\end{dow}
\medskip

Let us note that (A3$'$), (A4$'$) are minimal assumptions under
which there exists an $m$-a.e. finite solution of (\ref{eq2.5}).
In the next section we formulate some purely analytic conditions
on $f,\mu$ which for transient Dirichlet forms imply (A3$'$),
(A4$'$).

\begin{uw}
(i) A remarkable feature of Theorem \ref{th2.1} is that it can be
used in situations in which the underlying Dirichlet form is not
transient. For instance, it applies to Dirichlet problem with
Laplace operator in dimensions 1 and 2.\\
(ii) Suppose that $f$ does not depend on $y$ and $\mu\equiv 0$.
One of the equivalent conditions ensuring transiency of
$(\EE,D[\EE])$ is that for every nonnegative $f\in L^{1}(E;m)$
condition (A4$'$) is satisfied. This shows that if $d=1$ or $d=2$
then one can find $f\in L^{1}(E;m)$ such that there is no solution
of (\ref{eq2.5}).
\end{uw}

\begin{stw}
Let $u_{1}, u_{2}$ be solutions of \mbox{\rm(\ref{eq2.5})} with
the data $(f^{1},\mu_{1})$, $(f^{2},\mu_{2})$, respectively, such
that $u_{1}, u_{2}$ are of class \mbox{\rm(FD)}. Assume that
$\mu_{1}\le\mu_{2}$ and either $f^{1}(x,u_{1}(x))\le
f^{2}(x,u_{1}(x))$ $m$-a.e. and $f^{2}$ satisfies \mbox{\rm(H2)}
or $f^{1}(x,u_{2}(x))\le f^{2}(x,u_{2}(x))$ $m$-a.e. and $f^{1}$
satisfies \mbox{\rm(H2)}. Then $u_{1}(x)\le u_{2}(x)$ for q.e.
$x\in E$.
\end{stw}
\begin{dow}
By Theorem \ref{th2.1}, $u_{1}(X), u_{2}(X)$ are first components
of the solutions of BSDE$(\zeta,f^{1,0}+dA^{\mu_{1}})$ and
BSDE$(\zeta,f^{2,0}+dA^{\mu_{2}})$, respectively. Since
$(f^{1}_{u_{1}}-f^{2}_{u_{2}})^{+}=0$, $m$-a.e. and
$(\mu_1-\mu_2)^+=0$, then by uniqueness of the Revuz duality, for
q.e. $x\in E$,
$\int_{0}^{t}(f^{1}_{u_{1}}-f^{2}_{u_{2}})^{+}(X_s)\,ds=0$ and
$dA^{\mu_1}_t\le dA^{\mu_2}_t$, $t\ge 0,\, P_{x}$-a.s. It follows
that for q.e. $x\in E$ the solutions of the backward equations
satisfy on the space $(\Omega,\FF,\FF_t,P_x)$ the assumptions of
Proposition \ref{th1.1}. Therefore $u_{1}(X_{t})\le u_{2}(X_{t})$,
$t\ge 0$, $P_{x}$-a.s. for q.e. $x\in E$, and consequently,
$u_{1}(x)\le u_{2}(x)$ for q.e. $x\in E$.
\end{dow}

\nsubsection{Regularity of probabilistic solutions}
\label{sec5}

In this section we investigate regularity properties of
probabilistic solutions of (\ref{eq2.5}) under the additional
assumption that $(\EE,D[\EE])$ is transient and $\mu$ is a bounded
smooth measure. We also show that under these assumptions the
probabilistic solution of (\ref{eq2.5}) can be defined purely
analytically by duality.

We begin with definitions of some subsets of the set $S$ of smooth
measures. For more details we refer the reader to
\cite{Fukushima}.

$\MM_{0,b}$ denotes the class of all smooth measures on $E$ such
that $|\mu|(E)<\infty$, where $|\mu|$ stands for the total
variation of $\mu$ (elements of $\MM_{0,b}$ are sometimes called
soft measures; see \cite{DPP}). $\MM^+_{0,b}$ denotes the subset
of $\MM_{0,b}$ consisting of all nonnegative measures.

Let us recall that a Markovian semigroup $\{p_{t},t\ge 0\}$ is
called transient if for every nonnegative $f\in L^{1}(E;m)$,
\[
Gf(x)\equiv\lim_{N\rightarrow+\infty}\int_{0}^{N}p_{t}f(x)\,dt,
\quad m\mbox{-a.e.}
\]
(The limit above is well defined since the sequence on the
right-hand side is monotone). We say that a Dirichlet form $(\EE,
D[\EE])$ is transient if its associated semigroup is transient.

Assume that $(\EE,D[\EE])$ is a transient regular form. Then $\EE$
can be extended to a function space $\FF_{e}\subset\BB(E;m)$ in
such a way that $(\EE,\FF_{e})$ is a Hilbert space (see
\cite[Section 1.5]{Fukushima}). It is known that $\FF_{e}\cap
L^{2}(E;m)=D[\EE]$ and there exists a strictly positive bounded
function $g\in L^{1}(E;m)$ such that $\FF_{e}\subset
L^{1}(E;g\cdot dm)$ and
\begin{equation}
\label{eq5.4} (|u|,g)_{L^{2}(E;m)}\le \sqrt{\EE(u,u)},\quad
u\in\FF_{e}.
\end{equation}
In fact this is an equivalent condition for transiency of the form
$(\EE, D[\EE])$ (see \cite[Theorem 1.5.1]{Fukushima}). It is also
known (see \cite[Theorem 2.1.7]{Fukushima}) that any $u\in\FF_e$
admits a quasi-continuous modification that will always be
identified  with $u$.

By $S^{(0)}_{0}$ we denote the set of all nonnegative smooth
measures such that
\begin{equation}
\label{eq5.1} \int_{E}|v(x)|\,\mu(dx)\le c\sqrt{\EE(v,v)},\quad
v\in\FF_{e}\cap C_{0}(E)
\end{equation}
for some $c>0$. By Riesz's theorem, for every $\mu\in S^{(0)}_{0}$
there exists a unique function $U\mu\in\FF_{e}$, called the
($0$-order) potential of the measure $\mu$, such that
\[
\EE(U\mu,v)=\int_{E}v(x)\,\mu(dx),\quad v\in\FF_{e}\cap C_{0}(E).
\]
In fact, under our convention that elements of $\FF_e$ are
identified with their quasi-continuous modifications, the above
equality holds true for any $v\in\FF_e$ (see \cite[Theorem
2.2.5]{Fukushima}).

By $S_{0}$ we denote the class of nonnegative smooth measures of
finite energy integral, i.e. measures such that
\[
\int_{E}|v(x)|\,\mu(dx)\le c\sqrt{\EE_{1}(v,v)},
\quad v\in\FF\cap C_{0}(E).
\]
Again by Riesz's theorem, for every $\mu\in S_{0}$ and $\alpha>0$
there exists a unique function $U_{\alpha}\mu\in\FF$, called
$\alpha$-potential of $\mu$, such that
\[
\EE_{\alpha}(U_{\alpha}\mu,v)=\int_{E}v(x)\,\mu(dx),
\quad v\in\FF\cap C_{0}(E).
\]
Of course $S^{(0)}_{0}\subset S_{0}$.

By $S^{(0)}_{00}$ we denote the subset of $S^{(0)}_{0}$ consisting
of all measures $\mu$ such that $U\mu$ is bounded q.e. Note that
for every $\mu\in S$ there exists a generalized nest $\{F_n\}$
such that $\mathbf{1}_{F_{n}}\cdot|\mu|\in S^{(0)}_{00}$ (see
\cite[Theorem 2.2.4]{Fukushima} and remarks following the proof of
\cite[Corollary 2.2.2]{Fukushima}).

\begin{lm}
\label{lm5.1} Let $\mu\in S$, $\nu\in S^{(0)}_{00}$. Then for any
nonnegative Borel function $f$,
\[
E_{\nu}\int_{0}^{\zeta}f(X_{t})\,dA^{\mu}_{t} =\langle
f\cdot\mu,U\nu\rangle.
\]
\end{lm}
\begin{dow}
Let $\{F_{n}\}$ be a  generalized nest such that
$\mathbf{1}_{F_{n}}|f|\cdot|\mu|\in S^{(0)}_{00}$. By \cite[Lemma
5.1.3]{Fukushima}, for every $\alpha>0$,
\[
E_{x}\int_{0}^{\zeta}e^{-\alpha t}\mathbf{1}_{F_{n}}f(X_{t})\,
dA^{\mu}_{t}=U_{\alpha}(\mathbf{1}_{F_{n}} f\cdot\mu)(x)
\]
for q.e. $x\in E$. Hence
\[
E_{\nu}\int_{0}^{\zeta}e^{-\alpha t}\mathbf{1}_{F_{n}}
f(X_{t})\,dA^{\mu}_{t}=\langle
U_{\alpha}(\mathbf{1}_{F_{n}}f\cdot\mu),\nu\rangle =\langle
\mathbf{1}_{F_{n}}f\cdot\mu,U_{\alpha}\nu\rangle.
\]
Since $\mathbf{1}_{F_{n}}f\cdot\mu\in S^{(0)}_{0}$, applying
\cite[Lemma 2.2.11]{Fukushima} yields
\[
\langle \mathbf{1}_{F_{n}}f\cdot\mu,
U_{\alpha}\nu\rangle\rightarrow \langle
\mathbf{1}_{F_{n}}f\cdot\mu, U\nu\rangle.
\]
On the other hand, by the Lebesgue dominated convergence theorem,
\[
E_{\nu}\int_{0}^{\zeta}e^{-\alpha t}\mathbf{1}_{F_{n}} f(X_{t})\,
dA^{\mu}_{t}\rightarrow
E_{\nu}\int_{0}^{\zeta}\mathbf{1}_{F_{n}}f(X_{t})\,dA^{\mu}_{t}.
\]
Hence
\[
E_{\nu}\int_{0}^{\zeta}\mathbf{1}_{F_{n}}
f(X_{t})\,dA^{\mu}_{t}=\langle
\mathbf{1}_{F_{n}}f\cdot\mu,U\nu\rangle.
\]
Letting $n\rightarrow\infty$ in the above equality and using the
fact that $(\bigcup^{\infty}_{n=1}F_n)^c$ is exceptional we get
the desired result.
\end{dow}
\medskip

Let $\mathcal{A}$ denote the space of all quasi-continuous
functions $u:E\rightarrow\BR$ such that
$|\langle\nu,u\rangle|<\infty$ for every $\nu\in S^{(0)}_{00}$.
Let us stress that the space $\mathcal{A}$ depends on the form
$(\EE, D[\EE])$. Observe also  that $\FF_{e}\subset\mathcal{A}$.

The following definition may be viewed as an analogue of
Stampacchia's definition of a solution of linear elliptic equation
with measure data (see \cite{Stampacchia}).

\begin{df}
\label{def4.2} {\rm Assume that $(\EE,D[\EE])$ is transient and
$\mu\in\MM_{0,b}$. We say that $u:E\rightarrow\BR$ is a solution of
(\ref{eq2.5}) in the sense of duality if $u\in\mathcal{A}$,
$f_u\in L^1(E;m)$ and
\begin{equation}
\label{eq4.16}
\langle\nu,u\rangle=(f_u,U\nu)_{L^{2}(E;m)}+\langle\mu,U\nu\rangle,
\quad \nu\in S^{(0)}_{00}.
\end{equation} }
\end{df}

\begin{uw}
\label{uw5.3} Solutions in the sense of duality of linear nonlocal
elliptic equations with measure data are considered in
\cite{KarlsenPetittaUlusoy} in case $A=\Delta^{\alpha}$ on $\BRD$
with $\alpha\in (\frac12,1)$ and $d\ge2$. It is known (see
\cite[Exercise 2.2.1]{Fukushima}) that in this case
\[
Uf(x)=c(d,\alpha)\int_{\BRD}\frac{f(y)}{|x-y|^{d-2\alpha}}\,dy,
\quad f\in C_{0}(\BRD).
\]
From this one can easily deduce that $C_{0}^{+}(\BRD)\subset
S^{(0)}_{00}$. It follows in particular that if $u\in \mathcal{A}$
then  $u\in L^{1}_{loc}(E;m)$. It is also known (see
\cite[Exercise 1.5.2]{Fukushima}) that the form $(\EE,D[\EE])$
corresponding to $A$ is transient. Therefore in case $A$ has the
special form considered in \cite{KarlsenPetittaUlusoy} our
definition of a solution by duality agrees with the one introduced
in \cite{KarlsenPetittaUlusoy}.
\end{uw}

\begin{stw}
\label{prop444} Assume that $(\EE,D[\EE])$ is transient and
$\mu\in\MM_{0,b}$. If $u$ is quasi-continuous and $f_u\in
L^1(E;m)$, then $u$ is a probabilistic solution of
\mbox{\rm(\ref{eq2.5})} iff $u$ is a solution of
\mbox{\rm(\ref{eq2.5})} in the sense of duality.
\end{stw}
\begin{dow}
Let $u$ be a solution of (\ref{eq2.5}) in the sense of duality.
Let us denote by $w(x)$ the right-hand side of (\ref{eq2.7}) if it
is  finite and put $w(x)=0$ otherwise. By Proposition
\ref{prop3.1}, $w$ is finite $m$-a.e., and hence, by Lemma
\ref{lm2.2}, $w$ is quasi-continuous. By Lemma \ref{lm5.1},
$w\in\mathcal{A}$ and
\[
\langle\nu,w\rangle=(f_u,U\nu)_{L^{2}(E;m)}
+\langle\mu,U\nu\rangle, \quad \nu\in S^{(0)}_{00}.
\]
Thus, $\langle\nu,u\rangle=\langle\nu,w\rangle$ for $\nu\in
S^{(0)}_{00}$. By \cite[Theorem 2.2.3]{Fukushima}, this implies
that $u=w$ q.e. since $u,v$ are quasi-continuous.

Conversely, assume that $u$ is a probabilistic solution of
(\ref{eq2.5}). Then again by Lemma \ref{lm5.1}, $u\in\mathcal{A}$
and $u$ satisfies (\ref{eq4.16}).
\end{dow}
\medskip

In view of Proposition \ref{prop444} there arise natural
questions. When $f_{u}\in L^{1}(E;m)$? Is the assumption
$\mu\in\MM_{0,b}$, $f(\cdot,0)\in L^{1}(E;m)$ sufficient for
integrability of $f_u$? Is it always true that a probabilistic
solution $u$ of (\ref{eq2.5}) or a solution in the sense of
duality is locally integrable? We will show that  if
$\mu\in\MM_{0,b}$, $f(\cdot,0)\in L^{1}(E;m)$ then $f_{u}\in
L^{1}(E;m)$ but  $u$ need not be locally integrable.

Let $\mu$ be a Borel measure on $E$. In the sequel, $\|\mu\|_{TV}$
stands for its  total variation norm.

\begin{lm}
\label{lm.ap1} Assume that $(\mathcal{E},D[\mathcal{E}])$ is
transient, $\mu_{1}\in S$,  $\mu_{2}\in \MM^{+}_{0,b}$. If
\[
E_{x}\int_{0}^{\zeta} dA^{\mu_{1}}_t \le E_{x}\int_{0}^{\zeta}
dA^{\mu_{2}}_t
\]
for $m$-a.e. $x\in E$ then $\|\mu_{1}\|_{TV}\le \|\mu_{2}\|_{TV}$.
\end{lm}
\begin{dow}
By Lemma \ref{lm2.1} and \cite[Lemma 2.1.4]{Fukushima}, $
E_{x}\int_{0}^{\zeta}dA^{\mu_{1}}_t \le
E_{x}\int_{0}^{\zeta}dA^{\mu_{2}}_t$ for q.e. $x\in E$ and hence,
by Lemma \ref{lm5.1},
\begin{align}
\label{ap1} \langle \mu_{1}, U\nu\rangle\le \langle \mu_{2},
U\nu\rangle
\end{align}
for every $\nu\in S^{(0)}_{00}$. Since $E$ is locally compact and
$(\mathcal{E}, D[\mathcal{E}])$ is regular, there is a sequence
$\{U_{k}\}$ of decreasing open sets such that
$\mbox{cap}(U_{k})<\infty$ and $\bigcup_{k\ge 1}U_{k}=E$. Let
$e^{(0)}_{k}$ be the (0-order) equilibrium associated with the set
$U_{k}$ (see \cite{Fukushima} page 71). Then by the 0-order
counterpart of \cite[Lemma 2.1.1 ]{Fukushima} (see comments before
Lemma 2.1.8 in \cite{Fukushima}),  $0\le e^{(0)}_{k}\le 1$ q.e.,
$e^{(0)}(x)=1$ for q.e. $x\in U_{k}$, and
$e^{(0)}_{k}=U(\beta_{k})$, where $\beta_{k}\in S^{(0)}_{00}$ is
the measure associated with the 0-order potential $e^{(0)}_{k}$.
By (\ref{ap1}),
\[
\langle \mu_{1}, U(\beta_{k})\rangle
\le \langle \mu_{2}, U(\beta_{k})\rangle,\quad k\ge 1.
\]
Letting $k\rightarrow\infty$ and using Fatou's lemma gives the
desired result.
\end{dow}

\begin{stw}
\label{ap.stw1} Assume that $(\mathcal{E}, D[\mathcal{E}])$ is
transient, $\mu\in\MM_{0,b}$ and $f(\cdot,0)\in L^{1}(E;m)$. If
$u$ is a probabilistic solution of \mbox{\rm(\ref{eq2.5})}, then
$f_{u}\in L^{1}(E;m)$ and
\[
\|f_{u}\|_{L^{1}(E;m)}
\le\|f(\cdot,0)\|_{L^{1}(E;m)}+\|\mu\|_{TV}.
\]
\end{stw}
\begin{dow}
By Lemma \ref{lm0.2} and  Theorem \ref{th2.1},
\[
E_{x}\int_{0}^{\zeta} |f_{u}(X_{t})|\,dt \le E_{x}\int_{0}^{\zeta}
|f(X_{t},0)|\,dt + E_{x}\int_{0}^{\zeta} d|A^{\mu}|_{t}
\]
for $m$-a.e. $x\in E$. Therefore the desired inequality follows
from Lemma \ref{lm.ap1}.
\end{dow}

\begin{wn}
\label{wn.iwrty} If $(\EE,D[\EE])$ is transient and \mbox{\rm(A4)}
is satisfied, then $u$ is a probabilistic solution of
\mbox{\rm(\ref{eq2.5})} iff it is a solution of
\mbox{\rm(\ref{eq2.5})} in the sense of duality.
\end{wn}

\begin{prz}
\label{prz5.7} To show that in general a probabilistic solution of
(\ref{eq2.5}) is not locally integrable let us consider the
following trivial form
\[
\mathcal{E}(u,v)=\int_{-1}^{1}c(x)u(x)v(x)\,dx,\quad u,v\in
D[\mathcal{E}]=L^{2}(D;m),
\]
where $D=(-1,1)$, $c(x)=|x|$ and $m$ is the Lebesgue measure. Then
$(\mathcal{E},D[\mathcal{E}])$ is a transient regular Dirichlet
form and by Theorem \ref{th2.1} there exists a unique solution $u$ of
the equation
\[
-Au=1.
\]
Obviously, $u$ is given by the formula
\[
u(x)=|x|^{-1}\,,\quad
x\in D,
\]
and so is not locally integrable.
\end{prz}


\begin{uw}
\label{uw5.8}
Local integrability of $u$ is related to the
condition
\begin{equation}
\label{eq5.5}\forall K\subset E, K\mbox{-compact},\quad
U\mathbf{1}_{K}\in L^{\infty}(E;m).
\end{equation}
To see this, let us consider a transient regular Dirichlet form
$(\mathcal{E},D[\mathcal{E}])$. Suppose that for any $f\in
L^{1}(E;m)$ a solution $u$ of the problem
\[
-Au=f
\]
is locally integrable. Then by \cite[Theorem 5.1.3]{Fukushima},
for every compact $K\subset E$ and nonnegative $f\in L^{1}(E;m)$,
\[
\int_{K}|u|\,dm=\int_{K} u\,dm=(f,U\mathbf{1}_{K})_{L^2(E;m)}
<\infty,
\]
which implies that (\ref{eq5.5}) is satisfied. Conversely, assume
that (\ref{eq5.5}) is satisfied.  Let $u$ be a solution of the
problem (\ref{eq2.5}) with $f,\mu$ satisfying the assumptions of
Proposition \ref{ap.stw1}. Then applying  \cite[Theorem
5.1.3]{Fukushima} shows that for every compact $K\subset E$,
\[
\int_{K}|u|\,dm\le(|f_u|, U\mathbf{1}_{K})_{L^2(E;m)} +\langle
|\mu|,U\mathbf{1}_{K}\rangle,
\]
and hence (\ref{eq5.5}) is satisfied since $f_u\in L^1(E;m)$. Some
examples of forms satisfying (\ref{eq5.5}) will be given in
Section \ref{sec6}.
\end{uw}

\begin{stw}
\label{prop44} Assume that $(\EE,D[\EE])$ is transient and
$\mu\in\MM_{0,b}$. Then if $u$ is a solution of
\mbox{\rm(\ref{eq2.5})} and  $f_{u}\in L^{1}(E;m)$ then
$T_{k}(u)\in\FF_{e}$ for every $k\ge 0$. Moreover, for every $k\ge
0$,
\begin{equation}
\label{eq5.03} \EE(T_{k}(u),T_{k}(u))\le
k(\|f_{u}\|_{L^{1}(E;m)}+\|\mu\|_{TV}).
\end{equation}
\end{stw}
\begin{dow}
Let $\{F_{n}\}$ be a generalized nest such that
$\mathbf{1}_{F_{n}}|f_{u}|\cdot m+\mathbf{1}_{F_{n}}|\mu|\in
S^{(0)}_{0}$. Set
\[
u_{n}(x)=E_{x}\int_{0}^{\zeta}\mathbf{1}_{F_{n}}f_{u}(X_{t})\,dt
+E_{x}\int_{0}^{\zeta}\mathbf{1}_{F_{n}}(X_{t})\,
dA_{t}^{\mu},\quad x\in E
\]
and define $v_{n}, w_{n}$ as $u_n$ but with $f_{u},\mu$ replaced
by $f_{u}^{+},\mu^{+}$ and $f_{u}^{-},\mu^{-}$, respectively. Of
course, $u_{n}=v_{n}-w_{n}$. Set
$\mu_n^+=\mathbf{1}_{F_{n}}(f^{+}_{u}\cdot m+\mu^{+})$,
$\mu_n^-=\mathbf{1}_{F_{n}}(f^{-}_{u}\cdot m+\mu^{-})$. By Lemma
\ref{lm5.1} and \cite[Theorem 2.2.3]{Fukushima},
\[
v_{n}(x)=U\mu^+_n(x),\quad w_{n}(x)=U\mu_n^{-}(x)
\]
for q.e. $x\in E$. Hence $u_{n}\in\FF_{e}$, and consequently
$T_k(u_{n})\in\FF_{e}$, because $T_ku_n$ is a normal contraction
of $u_n$ and by \cite[Theorem 1.5.3]{Fukushima} every normal
contraction operates on $(\EE,\FF_e)$. Therefore
\begin{align*}
\label{eq5.3} \EE(u_{n},T_{k}(u_{n}))
=\int_ET_{k}(u_{n})(d\mu^+_n-d\mu^-_n) & \le
k(\|\mathbf{1}_{F_{n}}f_{u}\|_{L^{1}(E;m)}
+\|\mathbf{1}_{F_{n}}\mu\|_{TV})\\
&\le k (\|f_{u}\|_{L^{1}(E;m)}+\|\mu\|_{TV}).
\end{align*}
From the Beurling-Deny representation of the form $\EE$ (see
\cite[Theorem 3.2.1]{Fukushima}) it follows that
\[
\EE(T_{k}(u_{n}),T_{k}(u_{n}))\le \EE(u_{n},T_{k}(u_{n})).
\]
Hence
\[
\sup_{n\ge1} \EE(T_{k}(u_{n}), T_{k}(u_{n}))<\infty
\]
for every $k\ge0$. On the other hand, as in proof of (\ref{eq2.1})
one can show that $u_{n}\rightarrow u$ q.e. Therefore
(\ref{eq5.03}) follows from (\ref{eq5.4}) and the fact that
$(\EE,\FF_{e})$ is a Hilbert space.
\end{dow}

\begin{stw}
\label{prop5.5} Under the assumptions of Proposition \ref{prop44}
the following condition of vanishing energy is satisfied:
\begin{equation}
\label{eq5.6} \EE(\Phi_{k}(u),\Phi_{k}(u))\le\int_{\{|u|\ge k\}}
|f_{u}(x)|\,m(dx) +\int_{\{|u|\ge k\}}\,d|\mu|,
\end{equation}
where $\Phi_{k}(r)=T_{1}(r-T_{k}(r))$, $r\in\BR$.
\end{stw}
\begin{proof}
Let us define $u_n$ as in the proof of Proposition \ref{prop44}.
Then $\Phi_k(u_{n})\in\FF_e$ since $u_n\in\FF_e$ and $T_k$ is a
normal contraction for every $k\ge0$. Therefore
\[
\EE(u_{n},\Phi_{k}(u_{n}))
=(\mathbf{1}_{F_{n}}f_u,\Phi_{k}(u_{n}))_{L^{2}(E;m)} +\langle
\mathbf{1}_{F_{n}}\cdot\mu, \Phi_{k}(u_{n})\rangle.
\]
By the above equality and the definition of $\Phi_k$,
\[
\EE(u_{n},\Phi_{k}(u_{n})) \le\int_{\{|u_{n}|\ge k\}}
|f_{u}(x)|\,m(dx) +\int_{\{|u_{n}|\ge k\}}\,d|\mu|.
\]
Since $u_{n}\rightarrow u$ q.e. (see the proof of (\ref{eq2.1})),
it follows that
\[
\int_{\{|u_{n}|\ge k\}}|f_{u}(x)|\,m(dx) +\int_{\{|u_{n}|\ge
k\}}\,d|\mu|\rightarrow\int_{\{|u|\ge k\}}|f_{u}(x)|\,m(dx)
+\int_{\{|u|\ge k\}}\,d|\mu|.
\]
From the Buerling-Deny representation of the form $\EE$ (see
\cite[Theorem 3.2.1]{Fukushima} one can deduce that
\[
\EE(u_{n},\Phi_{k}(u_{n}))\ge\EE(\Phi_{k}(u_{n}),\Phi_{k}(u_{n})).
\]
Finally, by Proposition \ref{prop44},
\[
\EE(\Phi_{k}(u_{n}),\Phi_{k}(u_{n}))\rightarrow
\EE(\Phi_{k}(u),\Phi_{k}(u)),
\]
and the proof of (\ref{eq5.6}) is complete.
\end{proof}

\begin{uw}
From Proposition \ref{prop5.5} it follows in particular that if
$A$ is a uniformly elliptic divergence form operator on
$D\subset\BR^d$ with $d\ge3$ (i.e. $A$ corresponds to the form
$(\EE(D),D[\EE])$ defined by (\ref{eq4.0}) with coefficients
$a_{ij}$ satisfying (\ref{eq4.1})), then the probabilistic
solution of (\ref{eq2.5}) is a renormalized solution (see
\cite{BBGGPV}) of (\ref{eq2.5}), because in that case
$D[\EE]=\FF_{e}=H_{0}^{1}(D)$ by Poincar\'e's inequality. It is
worth pointing out that Propositions \ref{prop44} and
\ref{prop5.5} suggest possibility of extending the definition of
renormalized solutions to  general operators corresponding to
transient regular Dirichlet forms, notably to some  nonlocal
operators. Let us also note that renormalized solutions of some
elliptic equations with $L^1$-data and $A$ being a fractional
Laplacian on $\BR^d$ are studied in \cite{AAB}.
\end{uw}

\begin{uw}
Let $u$ be a solution of (\ref{eq2.5}). If $f_{u}\cdot m\in
S^{(0)}_0$ and $\mu\in S_{0}^{(0)}$ then by Lemma \ref{lm5.1} and
\cite[Theorem 2.2.5]{Fukushima}, $U(f_u+\mu)\in \FF_{e}$,
$u=U(f_u+\mu)$ q.e. and for every $v\in\FF_{e}$,
\[
\EE(u,v)=(f_{u},v)_{L^2(E;m)}+\langle v,\mu\rangle,
\]
i.e. $u$ is the usual weak solution of (\ref{eq2.5}).
\end{uw}

From Remark \ref{uwww} it follows that condition (A3) implies
(A3$'$). That (A4) implies (A4$'$) follows from the  proposition
given below.

\begin{stw}
\label{prop3.1} If $(\EE,D[\EE])$ is transient and
$\mu\in\MM^+_{0,b}$ then for $m$-a.e. $x\in E$,
\[
E_{x}\int_{0}^{\zeta}dA^{\mu}_{t}<\infty.
\]
\end{stw}
\begin{dow}
For $x\in E$ set
\[
S_{t}\mu(x)=E_{x}\intot dA^{\mu}_{s},\quad t\ge0,\qquad
G\mu(x)=\lim _{n\rightarrow\infty} S_{n}\mu(x).
\]
We have to prove that $G\mu(x)<\infty$ for $m$-a.e. $x\in E$. By
\cite[Theorem 5.1.3]{Fukushima} and the fact that the semigroup
$\{p_{t},\, t\ge 0\}$ associated with the form $\EE$ is Markovian,
\[
\|S_{t}\mu\|_{L^{1}(E;m)}=E_{m}\intot dA^{\mu}_{s} =\intot
\langle\mu, p_{s}1\rangle\,ds\le \intot \langle
\mu,1\rangle\,ds=t\|\mu\|_{TV}.
\]
We can now repeat the proof of \cite[Lemma 1.5.1]{Fukushima} with
$f\in L^{1}(E;m)$ replaced by $\mu$ and $S_{t}f$ replaced by
$S_{t}\mu$ to show that if there exists a strictly positive
function $g\in L^{1}(E;m)$ such that $Gg(x)<\infty$, $m$-a.e.,
then $G\mu(x)<\infty$, $m$-a.e. for every $\mu\in\MM^+_{0,b}$. But
such function $g$ exists since $(\EE,D[\EE])$ is transient.
\end{dow}

\begin{tw}
\label{th3.1} Assume that $(\EE,D[\EE])$ is transient and $\mu,f$
satisfy \mbox{\rm(A1)--(A4)}. Then there exists a unique
probabilistic  solution $u$ of \mbox{\rm(\ref{eq2.5})} such that
$u$ is of class \mbox{\rm(FD)} and $u\in\FD^{q}$, $q\in (0,1)$.
Moreover, $f_u\in L^{1}(E;m)$ and $T_{k}(u)\in \FF_{e}$ for every
$k\ge 0$.
\end{tw}
\begin{proof}
Follows from Proposition \ref{prop3.1} and Proposition
\ref{prop44}.
\end{proof}
\medskip

In view of Corollary \ref{wn.iwrty}, the solution $u$ of Theorem
\ref{th3.1} is a solution of (\ref{eq2.5}) in the sense of
duality.

Let $(\EE,D[\EE])$ be a regular Dirichlet form and let $g$ be a
strictly positive bounded Borel function on $E$. Then by
\cite[Lemma 1.6.6]{Fukushima} the perturbed form
$(\EE^{g},D[\EE])$, where
\[
\EE^{g}(u,v)=\EE(u,v)+(u,v)_{L^{2}(E;g\cdot dm)}
\]
is a transient regular Dirichlet form on $L^{2}(E;m)$. The
operator $A^g$ associated with $(\EE^{g},D[\EE])$ has the form
$A^g=A+g$, where $A$ is associated with $(\EE,D[\EE])$. Therefore
an immediate consequence of Theorem \ref{th3.1} is the following
proposition.

\begin{stw}
\label{wn3.111} If $\mu,f$ satisfy \mbox{\rm(A1)--(A4)} and $g$ is
a strictly positive bounded Borel function on $E$ then there
exists a unique probabilistic solution of the problem
\[
-Au+gu=f_u+\mu.
\]
\end{stw}

\nsubsection{Applications}
\label{sec6}

In this section we give typical examples of regular Dirichlet
forms  and indicate some situations in which our general results
are applicable. We keep the same assumptions on $E,m$ as in
Section \ref{sec5}.

Let $\{\nu_t,t>0\}$ be a symmetric convolution semigroup on
$\BR^d$ and let $\psi$ denote its L\'evy-Khintchine symbol, i.e.
for $x\in\BR^d$ we have
\[
\hat\nu_t(x)=\int_{\BR^d}e^{i(x,y)}\nu_t(dy)=e^{-t\psi(x)}.
\]
It is known (see \cite[Example 1.4.1]{Fukushima}) that the form
\[
\EE(u,v)=\int_{\BRD}\hat{u}(x)\hat{v}(x)\psi(x)\,dx,\quad u,v\in
D[\EE],
\]
\[
D[\EE]=\{u\in L^{2}(\BRD;dx);
\int_{\BRD}|\hat{u}(x)|^{2}\,\psi(x)\,dx<\infty\}
\]
determined by $\{\nu_t,t>0\}$ is a regular Dirichlet form on
$L^2(\BR^d;dx)$. Let us denote by $-\psi(\nabla)$ the nonpositive
self-adjoint operator associated with $(\EE,D[\EE])$.

\begin{stw}
\label{stw6.4} Assume that $\mu,f$ satisfy \mbox{\rm(A1)--(A4)}.
If $1/\psi$ is locally integrable on $\BRD$ (or, equivalently,
$\int_{0}^{\infty}\nu_{t}(K)\,dt<\infty$ for any compact set
$K\subset\BRD$), then there exists a unique probabilistic solution
of the problem
\[
-\psi(\nabla)u=f(x,u)+\mu,\quad x\in\BRD.
\]
\end{stw}
\begin{dow}
In \cite[Exercise 1.5.2]{Fukushima} it is shown that
$(\EE,D[\EE])$ is transient iff $1/\psi$ is locally integrable on
$\BR^d$ and that the last condition holds iff
$\int_{0}^{\infty}\nu_{t}(K)\,dt<\infty$ for any compact set
$K\subset\BRD$. Therefore the proposition follows from Theorem
\ref{th3.1}.
\end{dow}

\begin{prz}
\label{przyklad}
(i) (fractional Laplacian) Let $\psi(x)=c|x|^{\alpha}$ for some
$\alpha\in(0,2]$, $c>0$. The form is transient iff $\alpha<d$. Let
us also note that $\psi(\nabla)=c(\nabla^{2})^{\alpha/2}
=c\Delta^{\alpha/2}$. \medskip\\
(ii) (relativistic Schr\"odinger operator, see \cite{CMS}) Let
$\psi(x)=\sqrt{m^{2}c^{4}+c^{2}|x|^{2}}-mc^{2}$. It is an
elementary check that the form determined by $\psi$ is transient
if
$d\ge 3$.\medskip\\
(iii) (operator associated with the relativistic $\alpha$-stable
process). Let $0<\alpha<2$ and let
$\psi(x)=(|x|^{2}+m^{\alpha/2})^{2/\alpha}-m$. Then the associated
form is transient iff $d>2$ (see \cite[Chapter 5]{BBKRSV}).
\medskip\\
(iv) (operator associated with the variance gamma process). Let
$\psi(x)=\log(1+|x|^{2})$. Then the associated form is transient
iff $d>2$. This type of processes was applied in finance (see
\cite{MCC}).\medskip\\
(v) (operator associated with Brownian motion with Bessel
subordinator). Let
$\psi(x)=\log((1+|x|^{2})+\sqrt{(1+|x|^{2})^{2}-1})$. Then the
associated form is transient iff $d>1$ (see \cite[Chapter
5]{BBKRSV}).
\end{prz}

Let $(\EE,D[\EE])$ be the form of Proposition \ref{stw6.4} and let
$D$ be an open subset of $\BRD$. Set $L^2_D(\BR^d;dx)=\{u\in
L^2(\BR^d;dx):u=0\mbox{ a.e. on }D^c\}$, $D[\EE_D]=\{u\in
D[\EE]:\tilde u=0\mbox{ -q.e. on }D^c\}$, where $\tilde u$ is a
quasi-continuous version of $u$. By \cite[Theorem
4.4.3]{Fukushima}, the form $(\EE,D[\EE_D])$ is a regular
Dirichlet form on $L^2_D(\BR^d;dx)$, and by \cite[Theorem
4.4.4]{Fukushima}, if $(\EE,D[\EE])$ is transient then
$(\EE,D[\EE_D])$ is transient, too. Therefore from Theorem
\ref{th3.1} we get the following proposition.

\begin{stw}
\label{stw.tr} Let $D\subset\BR^d$ be an open set and $\mu,f$
satisfy \mbox{\rm(A1)--(A4)}. If $g:D\rightarrow\BR$ is a strictly
positive bounded Borel function or $1/\psi$ is locally integrable
on $D$ and $g$ is a nonnegative bounded Borel function then there
exists a unique probabilistic solution of the problem
\begin{equation}
\label{eq6.4} -\psi(\nabla)u+gu=f(x,u)+\mu, \quad u_{|D^{c}}=0.
\end{equation}
\end{stw}

Let $D$ be a domain in $\BRD$. Let us consider the Markovian
symmetric form on $D[\EE]=C^{\infty}_{0}(D)$ defined by
\begin{equation}
\label{eq4.0} \EE(u,v)=\sum_{i,j=1}^{d}\int_{D}a_{ij}(x)
\frac{\partial u}{\partial x_{i}}\frac{\partial v}{\partial
x_{j}}\,dx,
\end{equation}
where $a_{ij}$ are locally integrable functions on $D$ such that
for every $x\in D$ and $\xi\in\BRD$,
\begin{equation}
\label{eq4.1} \sum_{i,j=1}^{d} a_{ij}(x)\xi_{i}\xi_{j}\ge 0,\quad
a_{ij}(x)=a_{ji}(x),\quad 1\le i,j\le d.
\end{equation}
It is known (see \cite[Problem 3.1.1]{Fukushima}) that if one of
the following conditions
\begin{enumerate}
\item [(a)] $a_{ij}\in L^{2}_{loc}(D)$,
$\frac{\partial a_{ij}}{\partial x_{i}}\in L^{2}_{loc}(D)$, $1\le
i,j\le d$,
\item [(b)]there exists $\lambda>0$ such that
$\sum_{i,j=1}^{d}a_{ij}(x)\xi_{i}\xi_{j}\ge\lambda|\xi|^{2}$,
$x\in D$, $\xi\in\BRD$
\end{enumerate}
is satisfied, then the form $(\EE,D[\EE])$ is closable. Therefore
its smallest closed extension $(\bar{\EE},D[\bar{\EE}])$ is a
regular Dirichlet form on $L^{2}(D;dx)$ (see Theorems 3.1.1 and
3.1.2 in \cite{Fukushima}). Let us also note that if $d\ge 3$ and
condition (b) is satisfied then from \cite[Theorem
1.6.2]{Fukushima} and  the Gagliardo-Nirenberg-Sobolev inequality
it follows that $(\bar{\EE},D[\bar{\EE}])$ is transient (for other
conditions ensuring transiency see \cite[pp. 57--60]{Fukushima}).
Applying Theorem \ref{th3.1} we get existence of a solution of the
Dirichlet problem.

The following example  shows that $\mu$ in the definition of a
probabilistic solution of (\ref{eq2.5}) need not be Radon measure.
\begin{prz}
\label{prz6.4} Let $D$ be a bounded domain in $\BR^d$, $d\ge2$,
such that $0\in D$ and $U$ be an open ball with center at 0 such
that $\bar U\subset D$. Let us consider the form $(\EE_1,D[\EE])$,
where $(\EE,D[\EE])$ is the form defined by (\ref{eq4.0}) with
coefficients $a_{ij}$ of class $C^2_0(D)$ satisfying (\ref{eq4.1})
and  such that $a_{ij}=0$ on $\bar U$ for $i,j=1,\dots,d$. Let
$\BX=(X,P_x)$ be a diffusion corresponding to $(\EE,D[\EE])$. Then
the canonical subprocess $\BX^L=(X^L,P_x)$ of $\BX$ with respect
to the multiplicative functional $L_t=e^{-t}$, $t\ge0$, is a Hunt
process associated with $(\EE_1,D[\EE])$ (see \cite[Theorem
6.1.1]{Fukushima}. Let $\mu(dx)=g(x)\,dx$, where
\[
g(x)=|x|^{-\alpha}{\mathbf{1}}_U(x),\quad x\in\BR^d
\]
for some $\alpha>d$ and let
\[
A_t=\int^t_0g(X^L_s)\,ds,\quad t\ge0.
\]
It is easy to see that $P_x(X_t=x,t\ge0)=1$ if $x\in\bar U$ and
$P_x(X_t=X_{\sigma},t\ge\sigma)=1$ for $x\not\in\bar U$, where
$\sigma=\inf\{t\ge0:X\in U\}$. Therefore
$E_xA_{\infty}=E_x\int^{\infty}_0e^{-s}g(X_s)\,ds<\infty$ for
$x\in D\setminus\{0\}$. Consequently, $A$ is a PCAF of $\BX^L$
such that $E_xA_{\zeta}<\infty$ for a.e. $x\in D$. Of course,
$\mu$ is not Radon measure but $\mu\in S$, because $\mu$ is the
Revuz measure of $A$.
\end{prz}

\begin{stw}
\label{prop6.7} Let $D$ be a domain in $\BRD$ and let $a_{ij}$,
$1\le i,j\le d$, be measurable functions on $D$ satisfying
\mbox{\rm(\ref{eq4.1})}. Assume that $\mu,f$ satisfy
\mbox{\rm(A1)--(A4)} on $D$. If \mbox{\rm(a)} or \mbox{\rm(b)} is
satisfied and $g$ is a strictly positive bounded Borel function or
\mbox{\rm(b)} is satisfied, $d\ge3$ and $g$ is nonnegative, then
there exists a unique probabilistic solution of the problem
\begin{equation}
\label{eq6.3}
-\sum_{i,j=1}^{d}\frac{\partial}{\partial x_{i}}
(a_{ij}(x)\frac{\partial u}{\partial x_{j}})+gu =f(x,u)+\mu \mbox{
on }D,\quad u_{|\partial D}=0.
\end{equation}
\end{stw}

It is worth noting here that if $D$ is bounded and $a$ satisfies
(b) then a bounded signed measure $\mu$ on $D$ is of class
$\MM_{0,b}$ iff $\mu\in L^1(D;dx)+H^{-1}(D)$, where $H^{-1}(D)$ is
the space dual to $H^1_0(D)$ (see \cite{BGO}). Note also that the
obstacle problem for equations of the form (\ref{eq6.3}) and its
connection with BSDEs is investigated in \cite{RS}.

Theorem \ref{th3.1}  also applies to the Neumann problem. Let $D$
be a bounded domain in $\BRD$ with boundary of class $C$, i.e.
locally given by a continuous mapping. Let us consider the
Markovian symmetric form on $D[\EE]=C^{\infty}_{0}(\overline{D})$
defined by (\ref{eq4.0}) with $a_{ij}$ satisfying (\ref{eq4.1})
and condition (b) on $\overline{D}$. It is known (see
\cite[Example 1.6.1]{Fukushima}) that the form is closable and
$(\bar{\EE}, D[\bar{\EE}])=(\EE,H^{1}(D))$ is a regular Dirichlet
form on $L^2(\bar D;dx)$.

\begin{stw}
Let $D\subset\BR^d$ be a bounded domain of class $C$ and let
$a_{ij}$, $1\le i,j\le d$, be measurable functions on $\bar D$
satisfying \mbox{\rm(\ref{eq4.1})} and condition \mbox{\rm(b)}.
Assume that $\mu,f$ satisfy \mbox{\rm(A1)--(A4)} on $\bar D$ and
$g$ is a strictly positive bounded Borel function on
$\overline{D}$. Then there exists a unique probabilistic solution
of the problem
\[
-\sum_{i,j=1}^{d}\frac{\partial}{\partial x_{i}}
(a_{ij}(x)\frac{\partial u}{\partial x_{j}})u+gu =f(x,u)+\mu
\mbox{ on }D, \quad \frac{\partial u}{\partial n}=0 \mbox{ on }
\partial D.
\]
\end{stw}

\begin{uw}
(i) Let us consider the operator $\Delta^{\alpha/2}$,
$\alpha\in(0,2)$, on a bounded domain $D\subset\BRD$. Then for
every compact $K\subset D$,
\[
U\mathbf{1}_{K}(x)=E_{x}\int_{0}^{\zeta}
\mathbf{1}_{K}(X_{t})\,dt\le E_{x}\zeta\le E_{x}\tau_{B(r)} \le
c(d,\alpha)(r^{2}-|x|^{2})^{\alpha/2},\quad x\in B(r)
\]
where $X$ is an isotropic $\alpha$-stable L\'evy process on
$\BRD$, $D\subset B(r)=\{x\in\BRD;|x|\le r\}$ and
$\tau_{B(r)}=\inf\{t>0,X_{t}\notin B(r)\}$ (see, e.g., \cite{G}).
Accordingly, condition (\ref{eq5.5}) is satisfied. In fact, the
above inequalities show that $U1\in L^{\infty}(D;dx)$. Therefore,
if $f,\mu$ satisfy the assumptions of Proposition \ref{ap.stw1}
then $f_u\in L^1(E;dx)$, and consequently,
\[
\int_D|u|\,dm\le(|f_u|,U1)_{L^2(D;dx)} +\langle
|\mu|,U1\rangle<\infty.
\]
Thus, the solution of (\ref{eq6.4}) with $\psi(x)=|x|^{\alpha}$,
$\alpha<d$, $g\equiv 0$ belongs to $L^{1}(D;dx)$. The same
conclusion can be drawn for other operators of Example
\ref{przyklad} considered on bounded domain $D\subset\BR^d$ with
$d$ specified in the example. As above, to show this it suffices
to prove that $x\mapsto E_x\tau_{B(r)}$ is bounded on $D$. But the
last statement follows from results proved in \cite{Pruitt}.
\medskip\\
(ii) Let $D\subset\BR^d$, $d\ge3$, be a bounded domain and let $A$
corresponds to the form (\ref{eq4.0}) with coefficients $a_{ij}$
satisfying condition (b). Since it is know that in this case
$x\mapsto E_x\tau_D$ is bounded, then under the assumptions of
Proposition \ref{prop6.7} solutions of the problem (\ref{eq6.3})
are in $L^1(D;dx)$.
\end{uw}

Other interesting situations in which we encounter regular
Dirichlet forms include Laplace-Beltrami operators on manifolds
(see \cite{Fukushima}), quantum graphs (see \cite{Kuchment}),
perturbations of operators by Radon measures, Hamiltonians with
singular interactions (see \cite{BEKS,SV}), diffusion equations
with Wentzell boundary condition (see \cite{VV}).
\medskip\\
{\bf Acknowledgements}\\
Research supported by Polish Ministry of Science and Higher
Education (grant no. N N201 372 436).

\end{document}